\documentclass[11pt]{amsart}
\usepackage[utf8]{inputenc}
\usepackage[T1]{fontenc}
\usepackage{amsfonts,amsmath,amssymb,amsthm,tikz-cd}
\usepackage[shortlabels]{enumitem}
\usepackage[all]{xy}
\usepackage{lineno} % para numerar las lineas.

\usepackage[pagebackref]{hyperref}
  \hypersetup{
    colorlinks=true,
    linkcolor=cyan,
    citecolor=cyan
             }

\newcommand{\B}{\mathbb{B}} 
\newcommand{\C}{\mathbb{C}} 
 
\newcommand{\R}{\mathbb{R}}

\newcommand{\Z}{\mathbb{Z}}
\newcommand{\Q}{\mathbb{Q}}

\newcommand{\Pp}{\mathbb{P}} % Notation for projective spaces

\newcommand{\call}{\mathcal{L}} % Notation for compact faces of Newton polygons
\newcommand{\calo}{\mathcal{O}} % Notation for local rings
\newcommand{\calm}{\mathcal{M}} % Notation for sheaves of modules
\newcommand{\caln}{\mathcal{N}} % Notation for Newton polygons
\newcommand{\calp}{\mathcal{P}} % Notation for polynomials associated to compact edges of Newton polygons
\newcommand{\init}{\mathrm{in}} % Notation for initial part of a series

\newcommand{\supp}{\mathrm{Supp}} % Notation for supports of series
\newcommand{\trop}{\mathrm{Trop}} % Notation for local tropicalization
\newcommand{\ord}{\mathrm{ord}} % Notation for order of vanishing

\usepackage{todonotes}

\newtheorem{theorem}{Theorem}[section]
\newtheorem{lemma}[theorem]{Lemma}
\newtheorem{proposition}[theorem]{Proposition}
\newtheorem{corollary}[theorem]{Corollary}

\newtheorem{Atheorem}{Theorem}

\newtheorem{Btheorem}{Theorem}

\newtheorem{Ctheorem}{Theorem}

\theoremstyle{definition}
\newtheorem{definition}[theorem]{Definition}

\newtheorem{remark}[theorem]{Remark}

\title[A Eudoxian study of discriminant curves]{A Eudoxian study of discriminant curves \\associated to normal surface singularities}
\author[García Barroso]{Evelia Rosa García Barroso}
\author[Popescu-Pampu]{Patrick Popescu-Pampu}

\email{ergarcia@ull.es}
\email{patrick.popescu-pampu@univ-lille.fr}

\address{(Evelia Rosa Garc\'{\i}a Barroso) 
  Departamento de Matem\'aticas, Estad\'{\i}stica e I.O. Instituto Universitario de Matem\'aticas y Aplicaciones (IMAULL). Universidad de La Laguna. Apartado de Correos 456.
38200, La Laguna, Tenerife, Espa\~{n}a}
\address{(Patrick Popescu-Pampu) Universit\'e de Lille, CNRS, Laboratoire Paul Painlev\'e, 59000 Lille, France}


\thanks{\emph{Acknowledgments}. 
The authors gratefully acknowledge the support of Universidad de La Laguna (Tenerife, Spain), where part of this work was done (Spanish grant PID2023-149508NB-I00 funded by
MICIU/AEI/10.13039/501100011033 and by FEDER, UE). This work was also supported by ANR SINTROP (ANR-22-CE40-0014). The second author is grateful to the editors of MathSciNet for having sent him for reviewing the paper of Gryszka, Gwo\'zdziewicz and Parusi\'nski which made him learn about the theorem which we generalize here. We are grateful to Stephen Menn for his explanations on Eudoxus' contribution to the theory of proportion, as well as to Felix Delgado and Bernard Teissier for their answers to several questions. We thank also Fran\c{c}oise Michel and Wim Veys for their remarks on previous versions of this paper.}

\date{20 February 2026}

\begin{document}
%\linenumbers

\begin{abstract}
   Let $(f,g): (S,s) \to (\C^2, 0)$ be a finite morphism from a germ of normal  complex analytic surface to the germ of $\C^2$ at the origin. We show that the affine algebraic curve in $\C^2$ defined  by the initial Newton polynomial of a defining series of the discriminant germ of $(f,g)$ depends up to toric automorphisms only on the germs of curves defined by $f$ and $g$. This result generalizes a theorem of Gryszka, Gwo\'zdziewicz and Parusi\'nski, which is the special case in which $(S,s)$ is smooth. Our proof uses a common generalization of formulas of L\^e, Casas-Alvero and N\'emethi for the intersection number of the discriminant with a germ of plane curve. It uses also a theorem of Delgado and Maugendre characterizing the special members of pencils of curves on normal surface singularities. We apply it to the pencils generated by all pairs $(f^b, g^a)$, for varying positive integral exponents $a, b$, following a strategy initiated by Gwo\'zdziewicz and by Delgado and Maugendre. This is similar to the Eudoxian method of comparison of magnitudes by comparing the sizes of their positive integral multiples. 
\end{abstract}

\maketitle

\vspace{-1cm}

\tableofcontents

%%%%%%%%%%%%%%%%%%%%%%%%%%%%%%%%%%%%%%%%%%
\medskip
\section{Introduction}

Let $D \hookrightarrow (\C^2,0)$ be a non-zero germ of effective divisor. The {\em Newton polygon} of a defining series $\eta \in \C\{u,v\}$ of $D$ is independent of the choice of $\eta$.  We call it the {\em Newton polygon} $\caln(D)$ of $D$. Similarly, the {\em initial Newton polynomial} $\eta_{\caln}$ of $\eta$, that is, the restriction of $\eta$ to the union of compact faces of $\caln(D)$, is independent of the choice of $\eta$, {\em up to multiplication by a non-zero constant}. Therefore the zero locus of $\eta_{\caln}$ in the affine plane $\C^2$, seen as an effective divisor in which irreducible components are counted with multiplicities, is independent of the choice of the defining series $\eta$ of $D$. We call it the {\em initial Newton curve} of $D$. 

We apply this notion to the {\em discriminant divisor} $\Delta_{\varphi} \hookrightarrow (\C^2, 0)$ of a {\em finite} morphism $\varphi = (f, g) : (S,s) \to (\C^2, 0)$, where $(S,s)$ is a {\em normal surface singularity}, that is, a germ of normal complex analytic surface. Namely, our main theorem is:

\medskip
\noindent
{\bf Theorem}  \ref{thm:Athm}{\bf .}
    {\em Let $(S,s)$ be a normal surface singularity and let $\varphi= (f, g) : (S,s) \to (\C^2, 0)$ be 
    a germ of finite morphism. Then the initial Newton curve of its discriminant divisor $\Delta_{\varphi}$ depends 
    only on the divisors $Z(f), Z(g)$ defined by $f$ and $g$, 
    up to toric automorphisms of $\C^2$.} 
\medskip

This theorem generalizes Gryszka, Gwo\'zdziewicz and Paru\-si\'nski's \cite[Theorem 1.1]{GGP 22}, which concerned the case where $(S,s)$ is {\em smooth}. As in their paper, we prove first that if we replace the pair $(f,g)$ by a pair $((1 +  \alpha) f, (1 +  \beta) g)$ in which $\alpha, \beta$ are holomorphic germs of functions on $(S,s)$ vanishing at $s$, then the initial Newton curve of the discriminant remains unchanged. Theorem \ref{thm:Athm} is then a direct consequence of this invariance property. 
This theorem allows to define a new invariant of a pair of principal effective divisors on a normal surface singularity (see Definition \ref{def:initNcurvepair}): the corresponding initial Newton curve inside $\C^2$, well-defined up to toric automorphisms. 

Before the paper \cite{GGP 22}, discriminant divisors of finite morphisms $\varphi$ as above had been studied by many authors (see the papers \cite{M 82}, \cite{LMW 89}, \cite{D 94}, \cite{LW 97}, \cite{T 97}, \cite{M 99}, \cite{LMW 01}, \cite{DM 03}, \cite{KP 04}, \cite{M 08}, \cite{C 10}, \cite{DM 14}, \cite{MM 20}, \cite{DM 21}, \cite{DM 24} and their references), but only through their {\em discrete} invariants. What attracted us in \cite[Theorem 1.1]{GGP 22} is that it involved for the first time {\em moduli} of the situation. The previous chronological list of papers shows that the study of finite morphisms $\varphi$ progressed roughly from the case of {\em smooth} sources to arbitrary {\em normal} ones. In this paper we perform a similar progression concerning \cite[Theorem 1.1]{GGP 22}.

One of the main ingredients of the proof of \cite[Theorem 1.1]{GGP 22} is a formula of Casas-Alvero \cite[Theorem 3.1]{C 03} concerning the case when $(S,s)$ is {\em smooth}. It expresses the intersection number $Z(F) \cdot \Delta_{\varphi}$ of the discriminant divisor $\Delta_{\varphi}$ of $\varphi$ with a reduced principal divisor $Z(F)\hookrightarrow (\C^2,0)$ in terms of the degree $\deg(\varphi)$ of $\varphi$ and of the Milnor numbers $\mu(F), \mu(\varphi^* F)$ of $F$ and of its pullback $\varphi^* F$ to $(S,s)$. We adapt this ingredient to our context, by generalizing Casas-Alvero's formula to the case when {\em both the source and the target are arbitrary normal surface singularities}:

\medskip
\noindent
{\bf Theorem}  \ref{thm:Bthm}{\bf .} 
    {\em Let $\varphi : (S,s) \to (T,t)$ be a finite morphism between normal surface singularities, 
    with discriminant divisor $\Delta_{\varphi}$. If the holomorphic germ $F : (T,t) \to (\C, 0)$ 
    is such that both $F$ and $\varphi^* F$ are reduced, then 
           $\Delta_{\varphi} \cdot Z(F)   = (\mu(\varphi^* F) -1) - \deg(\varphi) \cdot (\mu(F) -1)$. 
    }
\medskip

It turns out that Theorem \ref{thm:Bthm} generalizes also L\^e's \cite[Lemma 4.4]{S 01}, which concerns the case when both $(T,t)$ and $Z(F)$ are smooth. Moreover, Casas-Alvero's formula is the special case of N\'emethi's \cite[Theorem A (a)]{N 91} for which $(S,s) = (\C^{n+1}, 0)$ and $(T,t) = (\C^2, 0)$.

Gryszka, Gwo\'zdziewicz and Paru\-si\'nski applied Casas-Alvero's formula to series $F \in \C\{u,v\}$ of the form $(v^k - t u^l)^N - u^{l(N + 1)}$, in which $t \in \C^*$ is a parameter, $k, l, N \in \Z_{>0}$ and $k,l$ are coprime. We simplify their approach, by applying Theorem B only to series of the form $u^b - \lambda v^a$, with $\lambda \in \C^*$, $a, b \in \Z_{>0}$ and $a, b$ coprime. 

We achieve this simplification using Delgado and Maugendre's \cite[Theorem 4]{DM 21}, which states an equivalence of two definitions of {\em special members} of pencils of curves associated to germs of finite morphisms $\varphi : (S,s) \to (\C^2, 0)$. Namely, we apply their theorem to {\em all} the finite morphisms $(f^b, g^a) : (S,s) \to (\C^2, 0)$, for varying pairs $(a,b)$ of coprime positive integers (see the proof of Proposition \ref{prop:invroots}). 

This is the reason why we qualify our approach as {\em Eudoxian}. A tradition tells that Eudoxus had the idea to develop a theory of proportions by comparing the sizes of all integer multiples $(b M_1, a M_2)$ of a given pair $(M_1, M_2)$ of magnitudes of the same kind. Similarly, we study the pair $(f,g)$ by looking simultaneously at all the pairs $(f^b, g^a)$. This idea was already used when $(S,s)$ is {\em smooth}: by Gwo\'zdziewicz in his study \cite{G 12} of Newton polygons of discriminants and by Delgado and Maugendre in their study \cite{DM 14} of combinatorial types of images of branches. But they mentioned no relations with Eudoxus' approach to proportion theory.

As a step in our proof of Theorem \ref{thm:Athm}, at the end of Section \ref{sec:combinvjac} we prove the following generalization of the main theorem of \cite{G 12}, which concerned the case when the source $(S,s)$ is {\em smooth}:

\medskip
\noindent
{\bf Theorem}  \ref{thm:Cthm}{\bf .}  
    {\em Let $(S,s)$ be a normal surface singularity and $\varphi= (f, g) : (S,s) \to (\C^2, 0)$ a germ 
    of finite morphism. Then the Newton polygon of its discriminant depends only on the 
    combinatorial type of the triple $(S, Z(f), Z(g))$.
    }
\medskip

This theorem is also a consequence of Michel's \cite[Theorems 4.8, 4.9]{M 08}, which did not concern Newton polygons explicitly and which she proved topologically. Our approach is completely different: we use Theorem \ref{thm:Bthm} and A'Campo's formula of \cite[Theorem 4]{A 75} for the Milnor number in terms of good resolutions.

\medskip
 This article is subdivided into sections. The first paragraph of every section describes its content. Each time a notation is introduced for the first time, it is written inside a $\boxed{\mathrm{box}}$.

Section  \ref{sec:background} is dedicated to notations and results of intersection theory. Section \ref{sec:specmemb} presents Delgado and Maugendre's  Theorem \ref{thm:DelMau} about the special members of pencils. In Section \ref{sec:genformCA} we prove Theorem \ref{thm:Bthm}.  In Section \ref{sec:Eudoxstrat} we explain for which reason we call the strategy of this paper {\em Eudoxian}. In Section \ref{sec:discfam} we start applying this strategy, by comparing the discriminant divisors of the finite morphisms $(f^b, g^a)$, for $(a,b)$ varying in $(\Z_{> 0})^2$. Its main result is Corollary \ref{cor:equiv}, which is an essential ingredient in the proof of Proposition \ref{prop:equivdist}, characterizing the special members of the pencils associated to the finite morphisms $(f^b, g^a)$. In Section \ref{sec:Npolint} we explain notations and properties of objects related to Newton polygons. Especially important from the viewpoint of the Eudoxian strategy is Proposition \ref{prop:intbin}, which expresses the tropical function of a series and the roots of its restrictions to the compact edges of its Newton polygon in terms of intersection numbers with binomial curves depending on $(a,b) \in (\Z_{> 0})^2$. In Section \ref{sec:tropfunctdisc} we apply Proposition \ref{prop:intbin} to the discriminant divisor $\Delta_{\varphi}$ of the finite morphism $\varphi$, which by combination with Theorem \ref{thm:Bthm} yields in Proposition \ref{prop:tropexpr} a formula expressing the tropical function of $\Delta_{\varphi}$ in terms of Milnor numbers of the functions $f^b - \lambda g^a$. In Section \ref{sec:combinvjac} we combine Proposition \ref{prop:tropexpr} with Delgado and Maugendre's Theorem \ref{thm:DelMau} in order to prove Theorem \ref{thm:Cthm}. In Section \ref{sec:compdivFitt} 
we show that for finite morphisms of smooth varieties of 
the same dimension, the discriminant divisor coincides with Teissier's discriminant subspace 
defined using Fitting ideals (see Proposition \ref{prop:equalspaces}) and we give 
a detailed proof of the fact that Fitting discriminants 
are compatible with base changes (see Proposition \ref{prop:compatdiscrbc}). 
Finally, in Section \ref{sec:invINC} we prove Theorem \ref{thm:Athm} by combining results of Sections \ref{sec:Npolint}--\ref{sec:compdivFitt}.

\medskip
\section{Background on normal surface singularities and their intersection theory}  
\label{sec:background}

In this section we fix the notations used throughout the paper concerning normal surface singularities and effective divisors on them and we explain intersection-theoretical results needed in subsequent proofs (see Propositions \ref{prop:projform}, \ref{prop:globint} and \ref{prop:quotform}). Note that we are interested only in intersection numbers of effective divisors at least one of which is {\em principal} (see Definition \ref{def:intnumb}). 
\medskip

Throughout the paper,  $\boxed{(S,s)}$ denotes a {\bf normal surface singularity}, that is, a germ of normal complex analytic surface. We denote by  $\boxed{\calo_{S,s}}$ its local ring of germs of holomorphic functions and by $\boxed{\mathfrak{m}_{S,s}}$ its maximal ideal. A {\bf branch} on $(S,s)$ is a reduced and analytically irreducible germ of curve on $(S,s)$. An {\bf effective divisor} on $(S,s)$ is a germ at $s$ of complex analytic effective divisor on a representative of the germ $(S,s)$, that is, an element of the free commutative monoid generated by the branches on $(S,s)$. 

If $\xi$ is either a germ of holomorphic function or of holomorphic differential form of degree two on $(S,s)$, we denote by $\boxed{Z(\xi)} \hookrightarrow (S,s)$ its corresponding effective divisor. That is, $Z(\xi)$ is the zero locus of $\xi$ in which each branch is counted with its generic multiplicity (the order of vanishing of $\xi$ along it in a pointed neighborhood of $s$ in a representative of $S$). As usual, the effective divisors of the form $Z(f)$ with $f \in \calo_{S,s}$ are called {\bf principal}.

In this paper we study the discriminant divisors of special types of {\em finite} morphisms:

\begin{definition}   \label{def:finmorph}
   A morphism $\psi : (X,x) \to (Y,y)$ between two germs of complex spaces is called {\bf finite} if it admits a representative which is finite in the global sense, that is, which is proper and with finite fibers. 
\end{definition}

One has the following classical characterisation of finiteness  
(see \cite[Theorem 3.4.24]{JP 00} for a more general statement):

\begin{proposition}   \label{prop:finchar}
   A morphism $\psi : (X,x) \to (Y,y)$ is finite if and only if it admits a representative such that $\psi^{-1}(y) = \{x\}$.
\end{proposition}

Effective divisors may be pushed forward by finite morphisms: 

\begin{definition}
    \label{def:dirimdiv}
   Let $\varphi : (S,s) \to (T,t)$ be a finite morphism between normal surface singularities and $D$ be a {\em branch} on $(S,s)$.  Then the effective divisor 
      \[\boxed{\varphi_* D} := \deg(\varphi|_{D}) \ \varphi(D) \]
   on $(T,t)$ is called the {\bf direct image} of $D$ by $\varphi$. Here $\varphi(D) \hookrightarrow (T,t)$ is the branch on $(T,t)$ defined as the reduced image of $D$ by $\varphi$, and $\deg(\varphi|_{D}) \in \Z_{ > 0}$ denotes the degree of the restriction $\varphi|_{D} : D \to \varphi(D)$ of $\varphi$ to $D$. Then,  the map $\varphi_*$ is extended by linearity to arbitrary effective divisors on $(S,s)$. 
\end{definition}

Let us recall standard terminology about {\em resolutions}:

\begin{definition}
    \label{def:termres}
     A {\bf good resolution} of the normal surface singularity $(S,s)$ is a proper bimeromorphic morphism 
   \[ \pi : (S_{\pi}, E_{\pi}) \to (S, s) \]
   whose {\bf exceptional divisor} $\boxed{E_{\pi}}:= \pi^{-1}(0)$ has normal crossings and  smooth irreducible components. If $h \in \mathfrak{m}_{S,s}$, a good resolution $\pi$ of $(S,s)$ is called a {\bf good resolution of $h$} if the {\bf total transform} $Z(\pi^* h)$ of $Z(h)$ is a divisor with normal crossings. The {\bf strict transform} of $Z(h)$ is the sum of irreducible components of the total transform which are not contained in $E_{\pi}$. 
\end{definition}

Using good resolutions, Mumford defined in \cite[Section II (b)]{M 61} an {\em intersection number} 
$D_1 \cdot D_2 \in \Q_{> 0}$ for pairs $(D_1, D_2)$ of non-zero effective divisors on $S$ 
without common branches. In this paper we use intersection numbers of effective divisors 
only when one of them is {\em principal} (see Theorem \ref{thm:Bthm}, Lemma 
\ref{lem:intbindiscr} and Lemma  \ref{lem:intpencilf}, in which the not necessarily principal divisor is either a critical or a discriminant divisor), as well as special properties of intersection numbers in this particular situation. Let us explain them. 

The following assymmetric definition of intersection numbers between a not necessarily 
principal divisor $D$ and a principal one $Z(f)$ amounts to compose normalizing parametrisations 
of the branches of $D$ with the germ $f$:

\begin{definition}    \label{def:intnumb}
    Let $(S,s)$ be a normal surface singularity, $D$ an effective divisor on it and 
    $f \in \mathfrak{m}_{S,s}$. 
       \begin{enumerate}
           \item   \label{brcase}
                If $D$ is a branch, consider a normalization $\boxed{n_D} : (\C, 0) \to (D,s)$ 
               of it. Then the {\bf intersection number of $D$ and $Z(f)$} is defined by:
                  \[  \boxed{D \cdot Z(f)} :=  \ord_x(n_D^*f), \]
                where $n_D^*f \in \C\{x\}$ denotes the pullback of $f$ by $n_D$ and 
                     $\boxed{\ord_x} : \C\{x\} \to \Z_{\geq 0} \cup \{\infty\}$ 
                denotes the $x$-adic valuation. 
           \item 
              If $D = \sum_{i \in I} D_i$ is a finite sum of branches $D_i$, then 
                  the {\bf intersection number of $D$ and $Z(f)$} is defined by:
                \[  \boxed{D \cdot Z(f)} := \sum_{i \in I}    D_i \cdot Z(f). \]
       \end{enumerate} 
\end{definition}

We end this section with three classical propositions. First, we need the following {\bf projection formula} in the proofs of Theorem \ref{thm:Bthm} and Lemma \ref{lem:intbindiscr}:

\begin{proposition}   \label{prop:projform}
     Let $\varphi : (S,s) \to (T,t)$ be a finite morphism between normal surface singularities, $D$  an effective divisor on $(S,s)$ and  $F \in  \mathfrak{m}_{T,t}$. Then 
        $  D \cdot Z(\varphi^*F) = \varphi_*D \cdot Z(F)$.
\end{proposition}

We use the following {\bf principle of conservation of number} in the proof of Theorem \ref{thm:Bthm}:

\begin{proposition}  \label{prop:globint}
      Let $(S,s)$ be a normal surface singularity, $D$ an effective divisor on it 
      and $f \in \mathfrak{m}_{S,s}$. Assume that 
      $D$ and $Z(f)$ have no common branch. Then there exists a 
      neighborhood $U$ of $s$ in a representative of $S$ and $V$ of $0$ in $\C$ such that 
      $F$ has a representative defined on the whole of $U$ and: 
         \[  D \cdot Z(f) = D \cdot_U Z( f - \varepsilon)  \]
      for every $\varepsilon \in V$, where $D \cdot_U Z( f - \varepsilon)$ denotes the global intersection number 
      of $D$ and of the divisor of $f - \varepsilon$ in $U$. 
\end{proposition}

When both divisors are principal, the apparently asymmetric definition of their intersection numbers yields in fact a symmetric result, due to the following fact, which we use in the proofs of Theorem \ref{thm:Bthm},  Lemma \ref{lem:intpencilf} and Proposition \ref{prop:tropexpr}: 

\begin{proposition}  \label{prop:quotform}
      Let $(S,s)$ be a normal surface singularity and $F, G \in \mathfrak{m}_{S,s}$. Assume that 
      $Z(F)$ and $Z(G)$ have no common branch. Then:
         \[  Z(F) \cdot Z(G) = \deg (F,G) = \dim \frac{\calo_{S,s}}{(F,G)},  \]
    where the first occurrence of the notation $(F,G)$ denotes the finite morphism $(F,G): (S,s) \to (\C^2, 0)$ and the second one denotes the ideal generated by $F$ and $G$ in the local ring $\calo_{S,s}$.
\end{proposition}

\begin{proof}
      The first equality is a consequence of Proposition \ref{prop:globint} and of the fact that $\deg (F,G)$ is the cardinal of a generic fiber of the morphism $(F,G)$, after choosing a convenient representative of it. 
      The second equality results from the fact that any normal surface singularity is Cohen-Macaulay and from \cite[Corollary D.6]{MN 20}.
\end{proof}

\medskip
\section{A theorem of Delgado and Maugendre on special members of pencils}
\label{sec:specmemb}

In this section we explain a theorem of Delgado and Maugendre which shows the equivalence of the two Definitions \ref{def:specialvalue} and \ref{def:speczones} of special members of the pencil $\Pp \varphi : (S,s) \dashrightarrow \C\Pp^1$ of germs of curves associated to a finite morphism $\varphi := (f,g) : (S,s) \to (\C^2, 0)$ (see Theorem \ref{thm:DelMau}).
\medskip

Let $(S,s)$ be a normal surface singularity.  Consider $f,g \in \mathfrak{m}_{S,s}$ such that the morphism 
   \[ \boxed{\varphi} := (f,g) : (S,s) \to (\C^2, 0)\] 
is finite (see Definition \ref{def:finmorph}). By Proposition \ref{prop:finchar}, this means that the effective divisors $Z(f)$ and $Z(g)$ have no common branch. We look at the critical and discriminant loci of $\varphi$ as effective divisors with possible non-reduced structures:

\begin{definition}
   \label{def:critdisc}
     The {\bf critical divisor} $\boxed{\Xi_{\varphi}}  \hookrightarrow (S,s)$ of $\varphi$ is the divisor of the differential form $df \wedge dg$. The {\bf discriminant divisor} $\boxed{\Delta_{\varphi}} \hookrightarrow (\C^2, 0)$ of $\varphi$ is the direct image $\varphi_*(\Xi_{\varphi})$ in the sense of Definition \ref{def:dirimdiv}. 
\end{definition}

One may associate to the finite morphism $\varphi$ not only two divisors as in Definition \ref{def:critdisc}, but also a {\em pencil} of curves, in which some members are distinguished:

\begin{definition}
    \label{def:specialvalue}
      Let $\boxed{\Pp}: \C^2 \dashrightarrow\C\Pp^1$ be the projectivisation map of the complex 
      vector space $\C^2$. Denote by $\boxed{\Pp \varphi} : (S,s) \dashrightarrow \C\Pp^1$ the map $ \Pp \circ \varphi$. We call it the {\bf pencil generated by $\varphi$}. For each $p \in \C\Pp^1$, denote by 
      $\boxed{L_p} \hookrightarrow \C^2$  the corresponding line,  equal to the closure of 
      $\Pp^{-1}(p)$ in $\C^2$. 
      Then $\varphi^*(L_p) \hookrightarrow (S,s)$, equal to the closure of 
      $(\Pp\varphi)^{-1}(p)$ in $(S,s)$, is the {\bf corresponding member of the pencil generated by $\varphi$}. A point $p \in \C\Pp^1$ is called 
     {\bf $\varphi$-distinguished} if the line $L_{p}$ is tangent at $0$ to the discriminant 
     $\Delta_{\varphi}$, that is, if and only if:
      \begin{equation} \label{eq:specar}
           \Delta_{\varphi} \cdot L_{p}>  \min \{ \Delta_{\varphi} \cdot L_q, \  q \in \C\Pp^1 \}.   
     \end{equation}
\end{definition}

If $p =[\alpha : \beta] \in \C\Pp^1$, then $L_p = Z(\beta u - \alpha v) \hookrightarrow \C^2_{u,v}$ is a 
global line and the corresponding member  $\varphi^*(L_p) = Z(\beta f - \alpha g)\hookrightarrow (S,s)$ of the pencil generated by $\varphi$ is an effective divisor on $(S,s)$. 
In the sequel we will concentrate on points $p$ whose homogeneous coordinates are of the form $[\lambda : 1]$ with 
$\lambda \in \C^*$ (see Proposition \ref{prop:equivdist}).

One may study the pencil generated by  $\varphi$ using particular kinds of good resolutions in the sense of Definition \ref{def:termres} of the ambient surface singularity $(S,s)$ (see \cite[Section 2]{LW 97} for the case when $(S,s)$ is smooth): 

\begin{definition}  
   \label{def:dicrit} 
    Let $\pi : (S_{\pi}, E_{\pi}) \to (S, s)$ be a good resolution of $(S,s)$. It is said {\bf to lift 
    the indeterminacies of $\Pp \varphi$} if the map 
    $\Pp \varphi_{\pi}:  = \Pp \varphi \circ \pi : S_{\pi} \dashrightarrow \C\Pp^1$  
    extends holomorphically to a neighborhood of $E_{\pi}$ in the whole surface $S_{\pi}$. In this case we still denote by 
    $\boxed{\Pp \varphi_{\pi}}$ this extension, seen as a germ along $E_{\pi}$. An irreducible component $E_i$ of the exceptional divisor 
    $E_{\pi}$ of $\pi$ is called {\bf $\Pp \varphi_{\pi}$-dicritical} 
    if the restriction of $\Pp \varphi_{\pi}$ to $E_i$ is surjective onto $\C\Pp^1$, that is, not constant.
\end{definition}

Having defined {\em dicritical} components allows to define {\em special points} of the parameter line $\C\Pp^1$ of the pencil:

\begin{definition}
  \label{def:speczones}
       Let $\pi : (S_{\pi}, E_{\pi}) \to (S, s)$ be a good resolution of $(S,s)$  
       which lifts the indeterminacies of $\Pp \varphi$. Denote by 
       $\boxed{D_{\pi}} \hookrightarrow E_{\pi}$  the union of the $\Pp \varphi_{\pi}$-dicritical 
       components of $E_{\pi}$. 
        A {\bf $\Pp \varphi_{\pi}$-special  point} or a {\bf $\varphi$-special point} is a point of $\C\Pp^1$ which is the image by $\Pp \varphi_{\pi}$ of a subset of $E_{\pi}$ of one of the following types:
           \begin{enumerate}
           \item  \label{sztype1}
              A non-$\Pp \varphi_{\pi}$-dicritical component of $E_{\pi}$.  
           \item \label{sztype2}
               A singular point of $D_{\pi}$. 
           \item \label{sztype3}
               A critical point of the restriction of $\Pp \varphi_{\pi}$ to a $\Pp \varphi_{\pi}$-dicritical component of $E_{\pi}$. 
         \end{enumerate}
\end{definition}

Note that if $E_i$ is not $\Pp \varphi_{\pi}$-dicritical, then $\Pp \varphi_{\pi}$ is constant in restriction to $E_i$, which shows that a $\Pp \varphi_{\pi}$-special point is indeed a {\em point} of the Riemann surface $\C\Pp^1$. Note also that we would have obtained the same set of special points if we had replaced $D_{\pi}$ by $E_{\pi}$ in condition \eqref{sztype2}. Indeed, any singular point of $E_{\pi}$ which is smooth on $D_{\pi}$ lies necessarily on a non-$\Pp \varphi_{\pi}$-dicritical component of $E_{\pi}$, therefore its image is also the image of that component, which shows that it is special, by condition \eqref{sztype1}. One may also characterize the special points as those points of $\C\Pp^1$ which do not have neighborhoods above which the restriction $\Pp \varphi_{\pi}|_{E_{\pi}} : E_{\pi} \to \C\Pp^1$ is a non-ramified covering.

We call the {\em $\Pp \varphi_{\pi}$-special} points also {\em $\varphi$-special} (removing the mentions of $\pi$ and $\Pp$) because of the following invariance property and of the fact that the map $\Pp \varphi$ is canonically determined by the germ $\varphi$:

\begin{proposition}
    \label{prop:invpispec}
    The set of $\Pp \varphi_{\pi}$-special points is independent of the choice of the good resolution $\pi$ of $(S,s)$ which lifts the indeterminacies of the map $\Pp \varphi$. 
\end{proposition}

\begin{proof}
   One may connect any two good resolutions with the stated property by a finite sequence of blow ups and blow downs of points. Therefore, it is enough to prove the statement for such a resolution $\pi$ and for the resolution $\pi \circ \pi_q$ obtained by composing $\pi$ with the blow up $\boxed{\pi_q}$ of $S_{\pi}$ at a point $q \in E_{\pi}$. One has then the commutative diagram:
    \begin{equation}   \label{eq:comdiagresol}
\begin{tikzcd}
     S_{\pi \circ \pi_q} \arrow[d, "\pi_q"'] 
         \arrow[dr, "\pi \circ \pi_q"] \arrow[drr, bend left, "\Pp \varphi_{\pi \circ \pi_q}"]
  &   &   \\
        S_{\pi}   \arrow[r, "\pi"']  \arrow[rr, bend right, "\Pp \varphi_{\pi}"']   &      S \arrow[r, dashed, "\Pp \varphi"']     &  \C\Pp^1
  \end{tikzcd}
\end{equation}
   The fact that $\pi$ lifts the indeterminacies of $\Pp \varphi$ implies that the same is true for $\pi \circ \pi_q$.

Following \cite[Page 692]{DM 21}, let us say that a {\em special zone of $\Pp \varphi_{\pi}$} is the closure of a  connected component of $E_{\pi} \setminus D_{\pi}$, a singular point of $D_{\pi}$ or a critical point of the restriction of $\Pp \varphi_{\pi}$ to a $\Pp \varphi_{\pi}$-dicritical component of $E_{\pi}$. 
By considering successively the three types of special zones, one sees that the $\Pp \varphi_{\pi \circ \pi_q}$-special zones are precisely the subsets of $E_{\pi \circ \pi_q}$ of the form $\pi_q^{-1}(X)$, where $X$ is a $\Pp \varphi_{\pi}$-special zone. As in this case $\Pp \varphi_{\pi \circ \pi_q}(\pi_q^{-1}(X)) = \Pp \varphi_{\pi}(X)$, we deduce that the sets of $\Pp \varphi_{\pi \circ \pi_q}$-special and of $\Pp \varphi_{\pi}$-special points coincide. 
\end{proof}

The following theorem is part of Delgado and Maugendre's \cite[Theorem 4]{DM 21}:  

\begin{theorem}  
   \label{thm:DelMau}
      Let  $\varphi : (S,s) \to (\C^2, 0)$ be a germ of finite morphism. 
      Let $\pi : (S_{\pi}, E_{\pi}) \to (S, s)$ be a good resolution of $(S,s)$  which lifts the 
      indeterminacies of $\Pp \varphi$ in the sense of Definition \ref{def:dicrit}. 
       Then a point $p \in \C\Pp^1$ is $\varphi$-distinguished in the sense of Definition 
       \ref{def:specialvalue} if and only if it is $\varphi$-special. 
\end{theorem}

Delgado and Maugendre proved their theorem when $\varphi$ is the minimal good resolution of $(S,s)$ which lifts the indeterminacies of $\Pp \varphi$. Proposition \ref{prop:invpispec} implies that it is also true in the greater generality of Theorem \ref{thm:DelMau}. We will use Theorem \ref{thm:DelMau} in the proof of Proposition \ref{prop:equivdist}.

%\vfill
%\pagebreak

\medskip
\section{A generalization of formulas of Casas-Alvero, L\^e and N\'emethi}
\label{sec:genformCA}

In this section we generalize to arbitrary normal surface singularities formulas of Casas-Alvero, L\^e and N\'emethi concerning finite morphisms between {\em smooth} germs of surfaces (see Theorem \ref{thm:Bthm}). The generalized formula expresses the intersection number of the discriminant divisor of a finite morphism between normal surface singularities with the divisor of a reduced function on the target in terms of the Milnor numbers of the function and its pullback by the finite morphism  and the degree of the morphism. We recall the definition of such Milnor numbers (see Definition \ref{def:Milnumber}) and a formula of A'Campo expressing them in terms of good resolutions (see Theorem \ref{thm:partACform}). Throughout this section, $\boxed{\chi(Y)}$ denotes the Euler-Poincar\'e characteristic of a topological space $Y$. 
\medskip

Consider a normal surface singularity $(T,t)$. Let $F \in \mathfrak{m}_{T,t}$ be reduced, which means that the principal divisor $Z(F)$ is reduced. One may define a {\em Milnor fibration} for it (a locally trivial fibration over a circle) following the model of \cite{M 68}. Namely, one first chooses  a representative of $(T,t)$ embedded in some $(\C^n, 0)$, then one  intersects it with a small enough compact euclidean ball centered at the origin, and one looks inside this intersection at the preimage by $F$ of a sufficiently small euclidean circle centered at the origin of $\C$ 
(see L\^e \cite{L 76}  for such a construction in arbitrary dimension and Cisneros-Molina and Seade 
\cite{CS 21} for a survey). A {\bf Milnor fiber} of $F$ is a fiber of a Milnor fibration of it. The following definition is then independent of the previous choices, and was explored initially in this generality by Bassein \cite{B 77}:

\begin{definition} 
   \label{def:Milnumber}
    Let $(T,t)$ be a normal surface singularity and $F \in \mathfrak{m}_{T,t}$. Its {\bf Milnor number} $\boxed{\mu(F)}$ is the first Betti number of the Milnor fibers of $F$ when $F$ is reduced and is $+ \infty$ otherwise. 
\end{definition}

\begin{remark}  \label{rem:BGdef}
     Bassein stated in \cite[Proof of Theorem (2.2)]{B 77} that 
      $\mu(F) = \dim (\omega / d \calo)$,  
      where $\calo$ denotes the local ring of the reduced curve singularity $(Z(F), 0)$ and $\omega$ 
      denotes its dualising sheaf. This allowed Buchweitz and Greuel to extend in \cite{BG 80} the notion of 
      Milnor number to arbitrary (not necessarily smoothable) reduced curve singularities, 
      by taking the expression 
      $\dim(\omega / d \calo)$ as its definition. In the sequel we will use the notion only 
      in smoothable settings, therefore Definition \ref{def:Milnumber} will be enough for us. 
\end{remark}

Milnor numbers may be expressed using good resolutions in the sense of Definition \ref{def:termres}. 
The following theorem is a special case of one of A'Campo's formulas of \cite[Theorem 4]{A 75}. A'Campo's theorem concerned holomorphic functions defined on germs of smooth spaces, but his proof extends readily to our context:

\begin{theorem}
    \label{thm:partACform}
    Let $(T,t)$ be a normal surface singularity and $F \in \calo_{T,t}$ be reduced. Let $\pi: (T_{\pi}, E_{\pi}) \to (T,t)$ be a good resolution of $F$. Denote by $E_{\pi}^m$ the subset of $E_{\pi}$ formed of the points in the neighborhood of which $Z(\pi^* F)$ is defined in suitable local coordinates $(u,v)$ by the equation $u^m =0$. Then:
      \[ \mu(F) = 1 - \sum_{m >0}m  \cdot  \chi(E_{\pi}^m). \]
\end{theorem}

We will use Theorem \ref{thm:partACform} in the proof of Proposition \ref{prop:combinvjnp} below. 
\medskip

Let now $\varphi : (S,s) \to (T,t)$ be a finite morphism between normal surface singularities, in the sense of Definition \ref{def:finmorph}. One may again associate to it {\em critical} and {\em discriminant divisors}. Definition \ref{def:critdisc} has to be slightly modified, as there is no more a canonical holomorphic differential form of degree two on the target to be pulled-back by $\varphi$ (the form $df \wedge dg$ being the pullback of the form $dx \wedge dy$ on the target $\C^2$ of the finite morphism of Definition \ref{def:critdisc}). Note first that in that definition, one may replace the form $dx \wedge dy$ by any non-vanishing germ of holomorphic form of degree two, which yields always the same critical divisor, thus also the same discriminant divisor. One may define therefore those divisors as follows in the enlarged setting:

\begin{definition}
    \label{def:critdivnorm}
    Let $\varphi : (S,s) \to (T,t)$ be a finite morphism between normal surface singularities. The {\bf critical divisor} $\boxed{\Xi_{\varphi}}  \hookrightarrow (S,s)$ of $\varphi$ is the closure of the critical divisor of the restriction of $\varphi$ to the complement of $s$ in a representative of $S$. The {\bf discriminant divisor} $\boxed{\Delta_{\varphi}} \hookrightarrow (T, t)$ of $\varphi$ is the direct image $\varphi_*(\Xi_{\varphi})$ in the sense of Definition \ref{def:dirimdiv}.
\end{definition}

Theorem \ref{thm:Bthm} extends to arbitrary normal surface singularities the following results: 
  \begin{itemize}
  \item Casas-Alvero's \cite[Theorem 3.2]{C 03}, in which both the source and the target are {\em smooth}. This theorem is the special case of N\'emethi's \cite[Theorem A (a)]{N 91} for which $(S,s) = (\C^{n+1}, 0)$ and $(T,t) = (\C^2, 0)$.
      \item  L\^e's \cite[Lemma 4.4]{S 01}, published in a paper by Snoussi,  which concerns the case when the target $(T,t)$ is smooth and the germ $F$ defines a smooth branch in $(T,t)$. 
  \end{itemize}

\begin{Btheorem} 
   \label{thm:Bthm}
    Let $\varphi : (S,s) \to (T,t)$ be a finite morphism between normal surface singularities, 
    with critical divisor $\Xi_{\varphi}$ and discriminant divisor $\Delta_{\varphi}$. 
    If $F \in \mathfrak{m}_{T,t}$ is such that both $F$ and $\varphi^* F \in \mathfrak{m}_{S,s}$ 
    are reduced, then:
       \[ \Delta_{\varphi} \cdot Z(F)=  \Xi_{\varphi} \cdot Z(\varphi^* F) = 
                 (\mu(\varphi^* F) -1) - \deg(\varphi) \cdot (\mu(F) -1).      \]
\end{Btheorem}

\begin{proof}
    Choose representatives of $(S,s)$ and $(T,t)$ by first intersecting $(T,t)$ with a sufficiently small ball, then by taking the total preimage of this representative by $\varphi$. For $\varepsilon \in \C^*$ with small enough $| \varepsilon|$, the preimages $F^{-1}(\varepsilon)$ and $(\varphi^* F)^{-1}(\varepsilon)$ taken inside those representatives are Milnor fibers of $F$ and $\varphi^* F$ respectively. 
    
    The restriction
     $ \boxed{\varphi_{\varepsilon}} : (\varphi^* F)^{-1}(\varepsilon) \to F^{-1}(\varepsilon)$ 
    of $\varphi$ to $(\varphi^* F)^{-1}(\varepsilon)$ is a ramified cover of Riemann surfaces with boundary. Denote by $\boxed{\rho} \in \Z_{\geq 0}$ the {\em total ramification index} of $\varphi_{\varepsilon}$. The Riemann-Hurwitz formula implies that:
      \begin{equation} \label{eq:RH}
         \chi((\varphi^* F)^{-1}(\varepsilon)) =  \deg(\varphi_{\varepsilon}) \cdot \chi(F^{-1}(\varepsilon)) - \rho. 
     \end{equation}
    We will write now differently the quantities appearing in the equality \eqref{eq:RH}.
     
    As a consequence of Propositions \ref{prop:globint} and \ref{prop:quotform}, we have: 
    \begin{equation}    \label{eq:eqdeg}
          \deg(\varphi_{\varepsilon}) =    \deg(\varphi).
     \end{equation}

     The germ $\varphi^* F$ being assumed reduced, we deduce that $\Xi_{\varphi}$ and $Z(\varphi^* F)$ have no common branch. By Proposition 
     \ref{prop:globint} applied to $f:= \varphi^* F$ and $D := \Xi_{\varphi}$, we see that for $\varepsilon \in \C^*$ with small enough absolute value:
                \[\Xi_{\varphi} \cdot Z(\varphi^* F) = \Xi_{\varphi} \cdot Z(\varphi^* F - \varepsilon) = \Xi_{\varphi} \cdot (\varphi^* F)^{-1}(\varepsilon).\]
    Moreover, we may ensure that the reduction $|\Xi_{\varphi}|$ of $\Xi_{\varphi}$ and the smooth curve $(\varphi^* F)^{-1}(\varepsilon)$ are transversal. If $p$ is one of their  intersection points, we may therefore choose local holomorphic coordinates $(x,y)$ on $S$ centered at $p$ such that $|\Xi_{\varphi}|= Z(y)$ and $(\varphi^* F)^{-1}(\varepsilon) = Z(x)$ in a neighborhood of $p$. Denote by $m \in \Z_{>0}$ the multiplicity of $\Xi_{\varphi}$ at $p$. Therefore $\Xi_{\varphi}= m Z(y)$ in a neighborhood of $p$. This implies the following value for the intersection number at the point $p$:
                \[ \Xi_{\varphi} \cdot_p (\varphi^* F)^{-1}(\varepsilon) = m. \]
    It implies also that after choosing suitable local coordinates on $T$ in a neighborhood of $\varphi(p)$, the germ of $\varphi$ at $p$ is given by $(x,y) \mapsto (x, y^{m+1})$. As $(\varphi^* F)^{-1}(\varepsilon) = Z(x)$ in a neighborhood of $p$, we deduce that $m$ is also the ramification index of $\varphi_{\varepsilon}$ at $p$. Summing over all such intersection points $p$, we get the equality: 
      \begin{equation}   \label{eq:constint}
            \rho = \Xi_{\varphi} \cdot Z(\varphi^* F).
     \end{equation}

     The surface $F^{-1}(\varepsilon)$ 
     is a Milnor fiber of $F$. This implies that $\mu(F)$ is its first Betti number. The Milnor fiber being connected, its $0$-th Betti number is $1$. As it is a connected surface with non-empty boundary, its higher Betti numbers are $0$. Therefore its Euler-Poincar\'e characteristic $\chi(F^{-1}(\varepsilon))$, which is by definition the alternating sum of Betti numbers, is: 
     \begin{equation}   \label{eq:chimu}
             \chi(F^{-1}(\varepsilon)) = 1 - \mu(F).
     \end{equation}

    An analogous reasoning proves that:
      \begin{equation}   \label{eq:chimubis}
            \chi((\varphi^* F)^{-1}(\varepsilon)) = 1 - \mu(\varphi^* F).
     \end{equation}

    Replacing the equalities \eqref{eq:eqdeg}--\eqref{eq:chimubis}  in equation \eqref{eq:RH}, we get the second equality of the statement. The first equality $\Delta_{\varphi} \cdot Z(F)  = \Xi_{\varphi} \cdot Z(\varphi^* F) $ is in turn a consequence 
    of the {\em projection formula} (see Proposition \ref{prop:projform}). 
\end{proof}

At the beginning of \cite[Section 3]{C 03}, Casas-Alvero comments that ``{\em In this section we give an algebraic proof of a formula [...] which may be understood as a local Riemann-Hurwitz formula if the Milnor numbers are viewed as numbers of vanishing cycles.}'' Then his proof of \cite[Theorem 3.2]{C 03} proceeds without using the Riemann-Hurwitz formula. 
By contrast, our proof of Theorem \ref{thm:Bthm} uses in an essential way the Riemann-Hurwitz formula for a suitable ramified cover of Riemann surfaces with boundary. In this sense, it is similar to L\^e's proof of \cite[Lemma 4.4]{S 01}.

In this paper we will apply Theorem \ref{thm:Bthm} only to {\em smooth} targets (see Proposition \ref{prop:tropexpr}). We proved nevertheless the theorem in greater generality because the proof is not more complicated and because we believe that it could be useful in this form in other contexts.

\medskip
\section{The Eudoxian strategy}
\label{sec:Eudoxstrat}

In this section we explain for which reason the strategy to study finite morphisms $(f,g)$ by looking simultaneously at all the finite morphisms $(f^b, g^a)$ for varying pairs $(a,b) \in \Z_{>0}^2$ may be called {\em Eudoxian}. At the end of the section we explain in which way we use this strategy in our proofs. 
\medskip

In Euclid's \cite[Book $V$]{Euc} is explained a theory of {\em proportionality} without passing through a definition of {\em ratio} of two magnitudes of the same kind (two lengths, two areas, two volumes, etc.) as a {\em number}. The core idea is that it is possible to decide when two pairs $(M_1, M_2)$ and $(M_3, M_4)$ of magnitudes, each of the same kind (but possibly of different kinds between the pairs, for instance two lengths and two areas) {\em are in the same ratio}. Namely, this is the case if for every $a,b \in \Z_{>0}$ and every $\boxed{\bowtie} \  \in \{=, <, >\}$, one has: 
  \[(b M_1 \bowtie a M_2) \ \Longleftrightarrow \ (b M_3 \bowtie a M_4). \]

  This approach to the theory of proportions is classically attributed to Eudoxus (see Menn \cite[Page 188]{M 18} for a discussion of this attribution). Several authors studied the relation between this approach and Dedekind's theory of {\em cuts} of $\Q$ or $\Q_{>0}$ (see Gardies \cite{G 84} and Corry \cite{C 94}). The main idea of this relation is that a pair $(M_1, M_2)$ of magnitudes of the same kind partitions the set $\Z_{>0}^2$ into the three subsets 
  $ \{ (a,b); \  (b M_1 \bowtie a M_2)\}$ 
  for $\bowtie$ varying in $\{=, <, >\}$, and that by looking at the sets of quotients $a/b$ corresponding to each subset one defines a cut of $\Q_{>0}$. 

In order to define either proportionality of two pairs of magnitudes or the cut associated to one pair it is important that the two magnitudes of each pair under consideration belong to a set on which are defined both {\em an action of the semiring $(\Z_{> 0}, +, \cdot)$} (which allows to take positive integer multiples of magnitudes) and {\em a total order} (which allows to compare two magnitudes).

Neither proportionality nor a cut of $\Q_{>0}$ can be defined if the magnitudes under consideration belong to a set endowed  with an action of $(\Z_{> 0}, +, \cdot)$, {\em but not with a total order}. This is precisely the case in our situation: we may think of two germs $f,g \in \mathfrak{m}_{S,s}$ as two magnitudes of the same kind. The positive integer $b \in \Z_{> 0}$ acts on $f$ transforming it into  $f^b$, but there is no natural total order on $\mathfrak{m}_{S,s}$. If we consider that the central idea attributed to Eudoxus is to compare two magnitudes by studying the pairs of their multiples {\em even in the absence of total order}, then we may indeed say -- as we do -- that studying the pairs $(f,g)$ using all the finite morphisms $(f^b, g^a) : (S, s) \to (\C^2, 0)$ is a {\em Eudoxian strategy}.

As far as we know, this strategy was initiated for {\em smooth} germs $(S,s)$ by Gwoz\-d\'ziewicz in \cite[Proof of Theorem 2.1]{G 12} and by Delgado and Maugendre in \cite[Introduction]{DM 14}.  It is in order to emphasize it for possible applications in other contexts that we decided to give it a name.

As $f^b$ and $g^a$ belong to a $\C$-algebra, one may compare them {\em $\C$-linearly} and {\em multiplicatively}. Both viewpoints are important in this paper. That is, we consider both the differences $f^b -\lambda g^a$, with $\lambda \in \C^*$ and the ratios $f^b/g^a$. The object which connects those differences and ratios is the pencil $\Pp \varphi_{a,b} : (S,s) \dashrightarrow \C\Pp^1$ generated by the finite morphism $\varphi_{a,b} := (f^b, g^a) : (S,s) \to (\C^2, 0)$ introduced in Section \ref{sec:discfam}. According to Definition \ref{def:specialvalue}, $Z(f^b -\lambda g^a)$ is the member of the pencil corresponding to the point 
 $[\lambda : 1] \in \C\Pp^1$ and $f^b/g^a$ is the expression of $\Pp \varphi_{a,b}$ in one of the two canonical charts of $\C\Pp^1$. The previous objects are used as follows in this paper:
   \begin{itemize}
       \item In Proposition \ref{prop:tropexpr}, we express the value $\trop^{\Delta_{\varphi}}(a,b)$ of the tropical function associated to the discriminant divisor $\Delta_{\varphi}$ of $\varphi := (f,g)$ in terms of the Milnor numbers $\mu(f^b -\lambda g^a)$ of the functions $f^b -\lambda g^a$. This proposition leads to Proposition \ref{prop:combinvjnp}, which in turn allows to prove Theorem \ref{thm:Cthm}. 
       \item  In Proposition \ref{prop:equivdist}, we express the set of roots of the restriction of a defining series $\eta_{\Delta_{\varphi}}$ of $\Delta_{\varphi}$ in terms of the set of special points of the pencil $\Pp \varphi_{a,b}$. This leads to a proof of Theorem \ref{thm:Athm}. 
   \end{itemize}

\medskip
\section{Relations between the discriminants of a family of finite morphisms}
\label{sec:discfam}

In this section we start applying the Eudoxian strategy described in Section \ref{sec:Eudoxstrat} by proving intersection-theoretic relations between the discriminant divisors of the finite morphisms $(f,g)$ and $(f^b, g^a)$, for $(a,b) \in \Z_{> 0}^2$ (see Proposition \ref{prop:inttwodiscr} and Corollary \ref{cor:equiv}). Those relations involve affine curves defined by binomials (see Definition \ref{def:bincurve}). 
\medskip

Assume as before that the morphism 
  $\varphi = (f,g) : (S,s) \to (\C^2, 0)$ is finite. As a consequence of Proposition \ref{prop:finchar}, the morphism
   \[   \boxed{\varphi_{a,b}} := (f^b, g^a) : (S,s) \to (\C^2, 0) \]
 is also finite, for every $(a,b) \in \Z_{> 0}^2$. Let $\boxed{\Xi_{a,b}} \hookrightarrow (S,s)$ and  
$\boxed{\Delta_{a,b}} \hookrightarrow (\C^2, 0)$  be the critical and discriminant divisors of $\varphi_{a,b}$.  
Note that $\varphi = \varphi_{1,1}$ and $\Xi_{a,b} = \Xi_{\varphi_{a,b}}$, 
$\Delta_{a,b} = \Delta_{\varphi_{a,b}}$ (see Definition \ref{def:critdisc}). We explain in Remark \ref{rem:reasonab} why we denote such pairs $(f^b, g^a)$ instead of $(f^a, g^b)$.

We will study the discriminant divisors $\Delta_{a,b}$ by intersecting them at the origin of $\C^2$ with the following affine curves defined by binomials:

\begin{definition}
  \label{def:bincurve}
    Let $(a,b) \in \Z_{> 0}^2$ and $\lambda \in \C^*$. The {\bf $(a,b,\lambda)$-binomial curve $\boxed{B_{a,b,\lambda}}$}  is the affine plane algebraic curve
      $ Z(u^b - \lambda v^a) \hookrightarrow \C^2_{u,v}.$
\end{definition}

One has the following equalities of  intersection numbers computed on the germ 
$(\C^2,0)$ (on the left sides) and on the germ $(S,s)$ (on the right sides): 

\begin{lemma}
  \label{lem:intbindiscr}
    Let $(a,b) \in \Z_{> 0}^2$ and $\lambda \in \C^*$. Then:
     \begin{enumerate}[(i)]
       \item  \label{intitem1}
          $\Delta_{1,1} \cdot B_{a,b, \lambda}  = \Xi_{1,1} \cdot Z(f^b - \lambda g^a) $. 
       \item  \label{intitem2}
          $\Delta_{a,b} \cdot B_{1,1, \lambda}  = \Xi_{a,b} \cdot Z(f^b - \lambda g^a) $. 
     \end{enumerate}
\end{lemma}

\begin{proof} The identity {\em (i)} results from the following sequence of equalities:
          \[\Delta_{1,1} \cdot B_{a,b, \lambda}  = 
            (\varphi_{1,1})_* (\Xi_{1,1}) \cdot B_{a,b, \lambda}  =  \Xi_{1,1} \cdot 
            (\varphi_{1,1})^*(B_{a,b, \lambda})  = \Xi_{1,1} \cdot 
          Z(f^b - \lambda g^a). \] 
        The first equality uses Definition \ref{def:critdisc} of the discriminant divisor as direct image of the critical 
        divisor, the second one uses the projection formula of Proposition \ref{prop:projform} and the last one uses the following consequence of Definition \ref{def:bincurve} of the binomial curve $B_{a,b, \lambda}$:
          \begin{equation}
              \label{eq:pb11}
               (\varphi_{1,1})^*(B_{a,b, \lambda}) = Z(f^b - \lambda g^a).
          \end{equation} 
        
    The identity {\em (ii)}  is proved similarly, using the following analog of the equality \eqref{eq:pb11}: 
          \begin{equation}
              \label{eq:pbab}
               (\varphi_{a,b})^*(B_{1,1, \lambda}) = Z(f^b - \lambda g^a).
          \end{equation}
\end{proof} 

Note that formulae \eqref{eq:pb11} and \eqref{eq:pbab} give two interpretations of the principal divisor $Z(f^b - \lambda g^a) \hookrightarrow (S,s)$ as a pull-back of a binomial curve by a finite morphism. The following lemma compares the critical divisors of these two finite morphisms $\varphi_{1,1}$ and $\varphi_{a,b}$:

\begin{lemma}
    \label{lem:compcrit}
    Let $(a,b) \in \Z_{> 0}^2$ and $\lambda \in \C^*$. Then:
    \[ \Xi_{a,b} = (b-1) Z(f) + (a-1) Z(g) + \Xi_{1,1}. \]
\end{lemma}

\begin{proof}
    One has:
     \[ d(f^b) \wedge d(g^a) = (bf^{b-1} df) \wedge (ag^{a-1}dg) = ab f^{b-1} g^{a-1} df \wedge dg.  \]
    Therefore: 
    \begin{eqnarray*}
     \Xi_{a,b} & = & Z(d(f^b) \wedge d(g^a)) = Z(ab f^{b-1} g^{a-1} df \wedge dg) 
     = (b-1) Z(f) + (a-1) Z(g) + Z(df \wedge dg) \\
      & = &  (b-1) Z(f) + (a-1) Z(g) + \Xi_{1,1}.
    \end{eqnarray*}  
\end{proof}

The effective divisor $Z(f^b - \lambda g^a) \hookrightarrow (S,s)$ appearing in Lemma \ref{lem:intbindiscr} has the following property, the intersection numbers being computed on $(S,s)$: 

\begin{lemma}
    \label{lem:intpencilf}
     Let $(a,b) \in \Z_{> 0}^2$ and $\lambda \in \C^*$. Then:
      \[  Z(f) \cdot Z(f^b - \lambda g^a) = a \ Z(f) \cdot Z(g).\]
\end{lemma}

\begin{proof} 
    Applying Proposition \ref{prop:quotform} twice, to the pairs $(f, f^b - \lambda g^a)$ and $(f, g^a)$, we get:
    \begin{eqnarray*}
        Z(f)  \cdot Z(f^b - \lambda g^a)  & = &   \dim \frac{\calo_{S,s}}{(f, f^b - \lambda g^a)} = 
           \dim \frac{\calo_{S,s}}{(f, \lambda g^a)}  \\
        & = & \dim \frac{\calo_{S,s}}{(f, g^a)}  = Z(f) \cdot Z(g^a)  = a \ Z(f) \cdot Z(g). 
    \end{eqnarray*}
\end{proof}

We proved Lemmas \ref{lem:intbindiscr}--\ref{lem:intpencilf} in order to get the following relation between the intersection numbers $ \Delta_{a,b} \cdot B_{1,1, \lambda} , \Delta_{1,1} \cdot B_{a,b, \lambda} $ (computed on $(\C^2,0)$) and $Z(f) \cdot Z(g)$ (computed on $(S,s)$):

\begin{proposition}
  \label{prop:inttwodiscr}
    Let $(a,b) \in \Z_{> 0}^2$ and $\lambda \in \C^*$. Then:
     \[ \Delta_{a,b} \cdot B_{1,1, \lambda} = (2ab - a - b) Z(f) \cdot Z(g) + \Delta_{1,1} \cdot B_{a,b, \lambda}. \]
\end{proposition}

\begin{proof}
         By Lemma \ref{lem:compcrit}:
         \[ \Xi_{a,b} \cdot Z(f^b - \lambda g^a)  =  ( (b-1) Z(f) + 
               (a-1) Z(g) + \Xi_{1,1}) \cdot Z(f^b - \lambda g^a). \]
         Develop the second member using the linearity of the intersection product. 
         By Lemma \ref{lem:intpencilf} and its analog for $Z(g)$, we get: 
         \begin{eqnarray*}
             \Xi_{a,b} \cdot Z(f^b - \lambda g^a)  & = & a(b-1)  Z(f) \cdot Z(g) + b(a-1) Z(g) \cdot Z(f)  
             +  \ \Xi_{1,1} \cdot Z(f^b - \lambda g^a)   \\ 
             & = & (a(b-1) + b(a-1)) Z(f) \cdot Z(g) + \Xi_{1,1} \cdot Z(f^b - \lambda g^a). 
         \end{eqnarray*} 
         Using now the two identities of Lemma \ref{lem:intbindiscr}, we obtain the desired equality. 
\end{proof}

As a consequence: 

\begin{corollary}
  \label{cor:equiv}
    Consider $\lambda_0 \in \C^*$. The following statements are equivalent:
      \begin{enumerate}
          \item $\Delta_{1,1} \cdot B_{a,b, \lambda_0}  > \min \{ \Delta_{1,1} \cdot B_{a,b, \lambda} , \  
              \lambda \in \C^* \}.$ 
          
          \item $\Delta_{a,b} \cdot B_{1,1, \lambda_0}  > \min \{ \Delta_{a,b} \cdot B_{1,1, \lambda}, \  
              \lambda \in \C^* \}.$
      \end{enumerate}
\end{corollary}

\begin{proof}
    The equivalence of the two statements results from Proposition \ref{prop:inttwodiscr} and from the fact that the product $(2ab - a - b) Z(f) \cdot Z(g)$ is independent of $\lambda \in \C^*$. 
\end{proof}

Corollary \ref{cor:equiv} will be used in combination with Proposition \ref{prop:intbin} in order to characterize in Proposition \ref{prop:equivdist} below the special points of the finite morphism $\varphi_{a,b} = (f^b, g^a)$.

\medskip
\section{Expressing tropical functions of divisors using intersection numbers}
\label{sec:Npolint}

This section is dedicated to the definitions of several notions associated to a bivariate series $\eta$: its {\em Newton polygon} $\caln(\eta) \in \C\{u,v\}$ (see Definition \ref{def:Newtpol}),   {\em local tropicalization} $\trop(\eta)$ (see Definition \ref{def:loctrop}), {\em tropical function} $\trop^{\eta}$ (see Definition \ref{def:tropfn}),  {\em initial Newton polynomial} $\eta_{\caln}$, {\em initial Newton curve} $Z(\eta_{\caln})$,  {\em restrictions} $\eta_{\call}$ and {\em univariate polynomials} $\calp_{\eta, \call}(z)$ associated to the compact faces $\call$ of $\caln(\eta)$ (see Definition \ref{def:initNP}). We explain that those objects depend only on the divisor $Z(\eta)$ defined by $\eta$ (see Proposition \ref{prop:depdivonly}) and  how to interpret $\trop^{\eta}$ and $\calp_{\eta, \call}(z)$ in terms of intersection numbers of $Z(\eta)$  with the binomial curves of Definition \ref{def:bincurve} (see Proposition \ref{prop:intbin}). Proposition \ref{prop:intbin} will be applied in Section \ref{sec:tropfunctdisc} to the discriminant divisors of finite morphisms $\varphi : (S,s) \to (\C^2, 0)$. 
\medskip

A vector $(a,b) \in \Z^2$ is called {\bf primitive} if $a$ and $b$ are coprime. 

Denote by $\boxed{M} = \Z^2$ the lattice of {\bf exponents of monomials} of the ring $\C\{u, v\}$ of convergent power series in the variables $u, v$. That is, to the vector $m= (p,q) \in M$ is associated the monomial $\boxed{\chi^m} = u^p v^q$. Let $\boxed{N} = \Z^2$ be its dual lattice of {\bf weight vectors}. That is, each vector $(a,b) \in N$ endows the variables $u, v$ with the weights $a$ and $b$ respectively. Let $\boxed{M_{\R}} := M \otimes_{\Z} \R, \ \boxed{N_{\R}} := N \otimes_{\Z} \R$ be the associated real vector spaces of dimension two. 
Their {\em non-negative quadrants} $\boxed{\sigma} \subset N_{\R}$ and $\boxed{\sigma^{\vee}} \subset M_{\R}$ consist of the vectors with non-negative coordinates. We denote by $\boxed{\mathring{\sigma}} $ and $\boxed{\mathring{\sigma}^{\vee}}$ their interiors, consisting of the vectors with positive coordinates.

A vector $(a,b) \in \Z^2$ may be seen both as an element of $M$ and of $N$. We use the two different notations $M$ and $N$ in order to specify each time in terms of which plane ($M_{\R}$ or $N_{\R}$) one has to think. For instance, the pairs $(a,b) \in \Z_{> 0}^2$ of Section \ref{sec:discfam} are to be thought 
as elements of $N \cap \mathring{\sigma}$.  

   The real plane $M_{\R}$ contains the {\em Newton polygons} of bivariate series:

\begin{definition}
 \label{def:Newtpol}
    Consider a non-zero series  
      \[ \eta = \sum_{m \in M \cap \sigma^{\vee}} \boxed{c_m} \chi^m \in \C\{u, v\}.\] 
      Its {\bf Newton polygon} $\boxed{\caln(\eta)}$ is the convex hull of the Minkowski sum $\sigma^{\vee} + \supp(\eta)$ of the first quadrant $\sigma^{\vee}$ of $M_{\R}$ and of the {\bf support} 
      $ \boxed{\supp(\eta)} := \{m \in M \cap \sigma^{\vee}, c_m \neq 0\} $
    of $\eta$. 
      The {\bf faces} of $\caln(\eta)$ are its vertices and its edges. 
    We {\em orient} each compact edge of $\caln(\eta)$ by increasing $p$-coordinate, 
    that is, using the orientation of its canonical projection to the axis of exponents of the variable $u$ (see Figure \ref{fig:orientNP}). 
\end{definition}

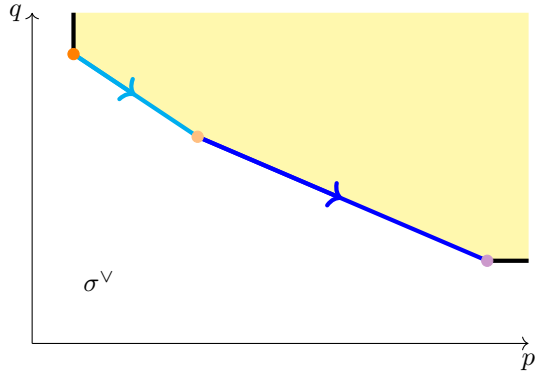
\begin{figure}[ht!] 
\begin{center}
\begin{tikzpicture}[x=0.55cm,y=0.55cm] 
% [x=0.6cm,y=0.6cm]
\tikzstyle{every node}=[font=\small]
\fill[fill=yellow!40!white] (0,5) --(3,3)-- (10,0) -- (11, 0) -- (11,6) -- (0,6) --cycle;

\draw[->] (-1,-2) -- (11,-2) node[right,below] {$p$};
\draw[->] (-1,-2) -- (-1,6) node[above,left] {$q$};

\node[above,right] at (0,-0.5) {$\sigma^{\vee}$};

\draw [ultra thick, color=black](0,5) -- (0,6);
\draw  [->] [ultra thick, color=blue](3,3) -- (6.5,1.5);
\draw [ultra thick, color=blue](10,0) -- (3,3);
\draw [ultra thick, color=cyan](3,3) -- (0,5);
\draw [->][ultra thick, color=cyan](0,5)--(1.5,4);

\draw [ultra thick, color=black](10,0) -- (11,0);

\node[draw,circle,inner sep=1.5pt,fill=violet!40, color=violet!40] at (10,0) {};
\node[draw,circle,inner sep=1.5pt,fill=orange!50,color=orange!50] at (3,3) {};
\node[draw,circle,inner sep=1.5pt,fill=orange, color=orange] at (0,5) {};

 \end{tikzpicture}
\end{center}
\caption{The orientations of the compact edges of a Newton polygon (see Definition \ref{def:Newtpol})}  
   \label{fig:orientNP}
    \end{figure}

Let $(a,b) \in N \cap \mathring{\sigma}$.

  \begin{definition} 
     \label{def:suppface}
      Let $(a,b) \in N \cap \mathring{\sigma}$. Look at it as the linear form 
  \[ \begin{array}{cccc}
         (a,b) : &   M_{\R}&  \to &  \R \\
                    &  (p,q) & \mapsto & ap + bq 
     \end{array} .  \] The {\bf support face $\boxed{\call(a,b)}$} of $(a,b)$ in $\caln(\eta)$ is the compact face of $\caln(\eta)$ on which the linear form $(a,b)$ achieves its minimum. 
  \end{definition}

     The support face $\call(a,b)$ is almost always a vertex, excepted when $(a,b)$ belongs to one 
     of the rays orthogonal to the compact edges of $\caln(\eta)$:

\begin{definition}
    \label{def:loctrop}
    The {\bf local tropicalization $\boxed{\trop(\eta)}$ of the series $\eta$} is the union of the rays orthogonal to the compact edges of the Newton polygon $\caln(\eta)$.  
\end{definition}

The second author and Stepanov introduced in \cite{PS 13} a general notion of {\em local tropicalization}  of subgerms of germs of toric varieties (see also the introductory text \cite{PS 25}). In the case of germs of hypersurfaces in $(\C^n,0)$, one can describe the local tropicalization in terms of the corresponding Newton polyhedron (see \cite[Proposition 11.8]{PS 13}). Definition \ref{def:loctrop} is a particular case of that description.

\begin{definition}
    \label{def:initNP}
    Let $\eta \in \C\{u, v\} \setminus \{0\}$. Its {\bf initial Newton polynomial} $\boxed{\eta_{\caln}} \in \C[u,v]$ is the sum of terms of $\eta$ which belong to the compact faces of $\caln(\eta)$. Its {\bf initial Newton curve} is the effective divisor $Z(\eta_{\caln}) \hookrightarrow \C^2$ of $\eta_{\caln}$. 
    
    Let $\call$ be a compact face of the Newton polygon $\caln(\eta)$. 
    The {\bf restriction $\boxed{\eta_\call}$ of $\eta$ to $\call$}  is the sum of terms of 
    $\eta$ whose exponents belong to $\call$.
\end{definition}

Kouchnirenko introduced  initial Newton polynomials in \cite[Definition 1.6]{K 76} under the name {\em partie principale newtonienne}, for series with any finite number of variables. He used them to formulate in \cite[Definition 1.9]{K 76} the important notion of {\em non-degenerate series}. Here it is for bivariate series:

\begin{definition}
    \label{def:Nnondeg}
    A series $\eta \in \C\{u, v\}\setminus \{0\}$ is called {\bf Newton non-degenerate} if its restrictions $\eta_{\call}$ to the compact faces $\call$ of $\caln(\eta)$ in the sense of Definition \ref{def:initNP} define reduced curves in the complex algebraic torus $(\C^*)^2$. 
\end{definition}

One may define also a univariate polynomial associated to any compact face of the Newton polygon of a bivariate series:

\begin{definition}
   \label{def:asspol}
    Let $\eta = \sum_{m \in M \cap \sigma^{\vee}} c_m \chi^m\in \C\{u, v\}\setminus \{0\}$  and 
    let $\call$ be a compact face of its Newton polygon $\caln(\eta)$. 
    Denote by $r \  \in \Z_{\geq 0}$ its {\bf integral length}, equal to the number of segments into which it is decomposed by its integral points.  Let $m_0, \ldots , m_r$ be those integral points, read as one travels on $\call$ in the positive sense (see Figure \ref{fig:subdivedgeNP}). The {\bf univariate polynomial $\calp_{\eta, \call}(z) \in \C[z]$ associated to $\eta$ and $\call$} is defined by:
      \[ \boxed{\calp_{\eta,\call}(z)}:= \sum_{j = 0}^r c_{m_j} z^j.\] 
\end{definition}

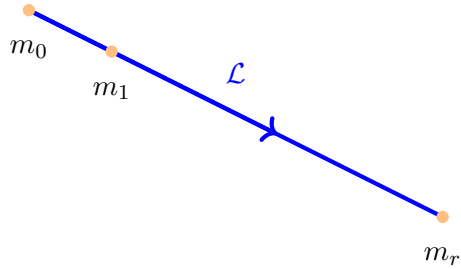
\begin{figure}[ht!] 
\begin{center}
\begin{tikzpicture}[x=0.55cm,y=0.55cm] 
% [x=0.6cm,y=0.6cm]
%\tikzstyle{every node}=[font=\small]

\draw [ultra thick, color=blue](0,5) -- (10,0);
\draw [->, ultra thick, color=blue](0,5) -- (6,2);

\node[draw,circle,inner sep=1.5pt,fill=orange!50,color=orange!50] at (0,5) {};
\node[draw,circle,inner sep=1.5pt,fill=orange!50,color=orange!50] at (2,4) {};
\node[draw,circle,inner sep=1.5pt,fill=orange!50,color=orange!50] at (10,0) {};

\node [below] at (0,4.5) {$m_0$};
\node [below] at (2,3.5) {$m_1$};
\node [below] at (10,-0.5) {$m_r$};
\node [above, color=blue] at (5,3) {$\call$};

 \end{tikzpicture}
\end{center}
\caption{The notations of Definition \ref{def:asspol}} 
   \label{fig:subdivedgeNP}
    \end{figure}

As $m_0$ and $m_r$ are the vertices of $\call$, the corresponding coefficients $c_{m_0}$ and $c_{m_r}$ are non-zero. That is, $\calp_{\eta, \call}(z)$ is a polynomial of degree $r$ with non-zero constant term, which shows that {\em its roots belong to $\C^*$}. This polynomial is constant if and only if $\call$ is a vertex of $\caln(\eta)$.

  The restriction $\eta_\call$ of $\eta$ to $\call$ is related to the polynomial $\calp_{\eta, \call}(z)$ as follows:

\begin{proposition}   
    \label{prop:linkrestrasspol}
    Let  $(a,b) \in N \cap \mathring{\sigma}$ be primitive. Consider its support face $\call(a,b)$ in $\caln(\eta)$ in the sense of Definition \ref{def:suppface}.  Then $\eta_{\call(a,b)} = \chi^{m_0} \ \calp_{\eta, \call(a,b)}(\chi^{(b, -a)})$.
\end{proposition}

\begin{proof}
   Recall from Definition \ref{def:asspol} that $m_0, \dots, m_r$ are all the integer points of $\call(a,b)$, read in the positive sense.  Thus, $m_j - m_0 = j(m_1-m_0)$ for all $j \in \{0, \dots, r \}$. Moreover, $m_1 - m_0 = (b, -a)$, because this vector is orthogonal to $(a,b)$ and that by Definition  \ref{def:asspol}, it is primitive and its first coordinate is positive. Therefore:
     \begin{eqnarray*}
        \label{eq:transfinit}
        \eta_\call & = & \sum_{j=0}^r c_j \chi^{m_j} = \chi^{m_0} \cdot \sum_{j=0}^r c_{m_j} \chi^{m_j- m_0}  = \chi^{m_0} \cdot \sum_{j=0}^r c_{m_j} \chi^{j(m_1- m_0)} \\
        & = &    \chi^{m_0} \cdot \sum_{j=0}^r c_{m_j} (\chi^{m_1- m_0})^j  
        =   \chi^{m_0} \cdot \sum_{j=0}^r c_{m_j} (\chi^{(b, -a)})^j = \chi^{m_0} \ \calp_{\eta, \call}(\chi^{(b, -a)}).
     \end{eqnarray*}   
\end{proof}

Proposition \ref{prop:linkrestrasspol} implies that (see also \cite[Proposition 1.4.20]{GBGPPP 20}):

\begin{proposition}
    \label{prop:Nndsimpleroots}
      A series $\eta \in \C\{u, v\}\setminus \{0\}$ is Newton non-degenerate if and only if for every compact face $\call$ of $\caln(\eta)$, the univariate polynomial $\calp_{\eta, \call}(z) \in \C[z]$ associated to $\eta$ and $\call$, in the sense of Definition \ref{def:asspol}, has simple roots. 
\end{proposition}

Note that whenever $\call$ is a vertex of $\caln(\eta)$, the conditions of Definition \ref{def:Nnondeg} and Proposition \ref{prop:Nndsimpleroots} are automatically satisfied. In particular, the product of a monomial and a unit of the ring $\C\{u,v\}$ is automatically Newton non-degenerate.

Each weight vector $w \in N \cap \mathring{\sigma}$ may be also seen as a valuation  of 
the local ring $\C\{u, v\}$:
    \[  \begin{array}{cccc}
            \C\{u, v\}  & \to & \Z_{\geq 0} \cup \{\infty\}  &  \\
                \eta    & \mapsto   & \boxed{w(\eta)} := \min \{w \cdot m, \   m \in  \supp(\eta)  \},  &
                      \mbox{if } \eta \neq 0,  \\
                 0   &  \mapsto   &  \infty.  &  
        \end{array}  \]

The following function, already introduced with this notation in \cite[Definition 1.4.4]{GBGPPP 20}, is classically known as the {\em support function} of the Newton polygon $\caln(\eta)$ (see Oda's \cite[Appendix A.3]{O 88}):

\begin{definition}
    \label{def:tropfn}
    Consider $\eta \in \C\{u, v\}\setminus \{0\}$. 
    The {\bf tropical function $\trop^{\eta}: \sigma \to \R_{\geq 0}$ of the series $\eta$} is defined by 
      $\boxed{\trop^{\eta}}(w) := w(\eta)$. 
\end{definition}

The local tropicalization $\trop(\eta)$ of Definition \ref{def:loctrop} is the non-differentiability set of the restriction of the tropical function $\trop^{\eta}$ to  $\sigma$ (this is a consequence of \cite[Proposition 1.4.8]{GBGPPP 20}; see 
\cite[Corollary A.19]{O 88} for a global analog in arbitrary finite dimension).

One has the following invariance properties of Newton polygons, local tropicalizations, tropical 
functions, initial Newton polynomials and restrictions of series to compact faces of their Newton  
polygons:

\begin{proposition}
    \label{prop:depdivonly}
       Let $D \hookrightarrow (\C^2,0)$ be an effective divisor and let $\eta \in \C\{u,v\}$ 
       be a defining function of it. Then:
        \begin{enumerate}
            \item  $\caln(\eta)$, $\trop(\eta)$ and $\trop^{\eta}$ are independent of the choice of 
                defining series $\eta$ of $D$. 
            \item Let $\call$ be a compact face of the Newton polygon  $\caln(\eta)$. 
            Then $\eta_{\caln}$, $\eta_{\call}$, $\calp_{\eta, \call}(z)$ are independent 
                of the choice of $\eta$, up to multiplication by elements of $\C^*$.
        \end{enumerate}
\end{proposition}

\begin{proof}  $\ $

{\em (1).}   One has the stronger property that 
     $\caln(\eta)$, $\trop(\eta)$ and $\trop^{\eta}$ depend only on the triple of effective 
     divisors $(Z(x), Z(y), Z(\eta))$  (see \cite[Proposition 1.4.13]{GBGPPP 20}). 

{\em (2).}  For every $w \in N \cap \mathring{\sigma}$, denote by $\init_w(\eta) \in \C[u,v]$ the {\bf $w$-initial part} of the series $\eta \in \C\{u, v\}$, equal to  
 the sum of terms of $\eta$ whose valuation is $w(\eta)$.   Let $\zeta \in \C\{u,v\}$ be such that $\zeta_0 := \zeta(0, 0) 
     \in \C^*$. One has therefore $\init_w (\zeta) = \zeta_0$ for all $w \in N \cap \mathring{\sigma}$.
     We get the equality 
        $ \init_w(\eta \cdot \zeta) =
         \init_w(\eta) \cdot \init_w(\zeta) =
         \zeta_0 \  \init_w(\eta)$. 
    The conclusion follows.   
\end{proof}

Proposition \ref{prop:depdivonly}   allows to formulate: 

 \begin{definition}
      \label{def:invobj}
   Let $D \hookrightarrow (\C^2,0)$ be an effective divisor. 
   Its {\bf Newton polygon} $\boxed{\caln(D)} \subset \sigma^{\vee}$, its {\bf local tropicalization} $\boxed{\trop(D)} \subset \sigma$, 
   its {\bf tropical function} $\boxed{\trop^D} : \sigma \to \R_{\geq 0}$, its {\bf initial Newton polynomial} 
   $\boxed{I_D} \in \C[u,v]$ and its {\bf initial Newton curve} $\boxed{Z(I_D)} \hookrightarrow \C^2$ are the corresponding objects associated to any defining function of $D$. 
   
   Let $\call$ be a compact face of the Newton polygon  $\caln(D)$. We define similarly 
   the {\bf restriction $\boxed{I_{D,\call}} \in \C[u,v]$ of $I_D$ to $\call$} and the {\bf univariate polynomial 
   $\boxed{\calp_{D,\call}} \in \C[z]$}, well-defined up to multiplication by constants. 
\end{definition}

Note that $\trop^D$ depends in fact only on the Newton polygon $\caln(D)$. This allows to speak of the {\bf tropical function $\boxed{\trop^{\caln}} : \sigma \to \R_{\geq 0}$ of a Newton polygon} $\caln \subseteq \sigma^{\vee}$.

A Newton polygon $\caln$ contains the same information as its tropical function $\trop^{\caln}$. Indeed, it may be reconstructed as follows from it (see \cite[Proposition 1.4.7 and the paragraph following it]{GBGPPP 20} or \cite[Theorem A.18]{O 88} for a global analog):

\begin{proposition}
    \label{prop:reconstr}
    Let $\caln \subseteq \sigma^{\vee}$ be a Newton polygon. Then:
      \[  \caln =  \{ m \in \sigma^{\vee}, \   w\cdot m \geq \trop^{\caln}(w) \mbox{ for every } w \in \sigma\}. \]
\end{proposition}

Consider again a series $\eta \in \C\{u,v\}$. In our context, the advantage of its tropical function $\trop^{\eta}$ compared to its Newton polygon $\caln(\eta)$ is that $\trop^{\eta}$ may be interpreted in terms of  intersection numbers with binomial curves, as shown by the following special case of \cite[Proposition 1.4.12 and the paragraph following it]{GBGPPP 20} (recall that the binomial curves $B_{a,b,\lambda}$ were introduced in Definition \ref{def:bincurve}):

\begin{proposition}
    \label{prop:intbin}
    Let $\eta \in \C\{u, v\}$. Assume that $(a,b) \in N \cap \mathring{\sigma}$ is primitive and let $\call(a,b)$ be its support face in $\caln(\eta)$, in the sense of Definition \ref{def:suppface}. 
    \begin{enumerate}
      \item \label{eq:itemeq}
         One has the equality 
        $ \min \{Z(\eta) \cdot B_{a,b,\lambda} , \lambda \in \C^*\} = \trop^{\eta}(a,b)$. 
      \item \label{eq:itemineq}
        The inequality 
      $ Z(\eta) \cdot B_{a,b,\lambda} > \trop^{\eta}(a,b)$ 
    holds if and only if $\calp_{\eta, \call(a,b)}(\lambda)=0$, where  $\calp_{\eta, \call(a,b)}(z)$ is the polynomial associated to $\eta$ and $\call(a,b)$ as in Definition \ref{def:asspol}. 
    \end{enumerate}
\end{proposition}

Proposition \ref{prop:intbin} will be used in combination with Corollary \ref{cor:equiv} in order to relate in Proposition \ref{prop:equivdist} below the univariate polynomial $\calp_{\Delta_{1,1}, \call(a,b)}$ and the special points of the finite morphism $\varphi_{a,b} = (f^b, g^a)$.

\medskip
\section{Application to the tropical functions of discriminant divisors}
\label{sec:tropfunctdisc}

In this section we come back to a finite morphism $\varphi=(f,g) : (S,s) \to (\C^2, 0)$ and to the associated finite morphisms $\varphi_{a,b} = (f^b, g^a)$ of the Eudoxian strategy explained in Section \ref{sec:Eudoxstrat}. We apply Proposition \ref{prop:intbin} to $Z(\eta) = \Delta_{\varphi} = \Delta_{1,1}$.   Namely, in Proposition \ref{prop:equivdist} we characterize the $\varphi_{a,b}$-special points of $\C\Pp^1$ using the restrictions of the discriminant divisor $\Delta_{1,1}$ of $\varphi$ to the compact edges of the Newton polygon of $\Delta_{1,1}$, called the {\em jacobian Newton polygon of $\varphi$}. Then, in Proposition \ref{prop:tropexpr} we express the tropical function $\trop^{\Delta_{1,1}}$ of this discriminant divisor in terms of the Milnor numbers of the functions $f^b - \lambda g^a$. 
\medskip

 Teissier introduced the following terminology in \cite[Section 5.17]{T 77} for finite morphisms with {\em smooth} sources (see also \cite{T 80}): 

\begin{definition}  \label{def:JNpol}
  The {\bf Jacobian Newton polygon} of the finite morphism $\varphi$ is the Newton polygon $\caln(\Delta_{\varphi})$ of the discriminant divisor $ \Delta_{\varphi}$ of $\varphi$. 
\end{definition}

The following result characterizes the roots of the univariate polynomials associated to the compact faces of the Jacobian Newton polygon of $\varphi$ in terms of the special points of the finite morphisms $\varphi_{a,b}$:

\begin{proposition}
    \label{prop:equivdist}
       Fix $\lambda \in \C^*$. 
    Let $(a,b) \in N \cap \sigma^{\circ}$ be primitive. Denote by $\call(a,b)$ the support face of $(a,b)$ in $\caln(\Delta_{1,1})$, in the sense of Definition \ref{def:suppface}.  Then
    $\calp_{\Delta_{1,1}, \call(a,b)}(\lambda) = 0$ if and only if $[\lambda : 1]$ is a $\varphi_{a,b}$-distinguished point of $\C\Pp^1$ in the sense of Definition \ref{def:specialvalue}. 
\end{proposition}

\begin{proof} Recall from Definition \ref{def:specialvalue} that  $L_{[\lambda\ : \ 1]}$ denotes the line $Z(u- \lambda v) \hookrightarrow \C^2_{u,v}$, for every $\lambda \in \C^*$. One has the following equivalences:
    \begin{eqnarray*}
        \calp_{\Delta_{1,1}, \call}(\lambda)=0 & \Longleftrightarrow & \Delta_{1,1} \cdot B_{a,b, \lambda}  > \min \{ \Delta_{1,1} \cdot B_{a,b, \lambda'}, \  \lambda' \in \C^* \} \\ 
      & \Longleftrightarrow & 
         \Delta_{a,b} \cdot B_{1,1, \lambda}  > \min \{ \Delta_{a,b} \cdot B_{1,1, \lambda'}, \  \lambda' \in \C^* \} \\
      & \Longleftrightarrow & 
         \Delta_{a,b} \cdot L_{[\lambda \ : \ 1]}  > \min \{ \Delta_{a,b} \cdot L_{[\lambda'\ : \ 1]}, \  \lambda' \in \C^* \} \\
      & \Longleftrightarrow & 
         [\lambda:1] \mbox{ is } \varphi_{a,b}-\mbox{distinguished}.
    \end{eqnarray*}
    Indeed: 
      \begin{itemize}
        \item the first equivalence results from Proposition \ref{prop:intbin} applied to $Z(\eta) = \Delta_{1,1}$;

          \item the second one results from Corollary \ref{cor:equiv};

          \item the third one results from the fact that 
             $B_{1,1, \lambda'} = Z(u - \lambda' v) = L_{[\lambda'\ : \ 1]}$ for every $\lambda' \in \C^*$; 

          \item the fourth one results from Definition \ref{def:specialvalue}. 

      \end{itemize}
\end{proof}

Proposition \ref{prop:equivdist}  will be used in the proofs of Propositions \ref{prop:combinvjnp} \eqref{eq:notintrop} and  \ref{prop:invroots} below.

\medskip
The following result expresses the tropical function of the  discriminant divisor $\Delta_{1,1}$ of the finite morphism $\varphi_{1,1} =(f,g)$ in terms of the Milnor numbers of the members of the pencil determined by the finite morphism $\varphi_{a,b} =(f^b,g^a)$: 

\begin{proposition}
     \label{prop:tropexpr}
    Let $(a,b) \in \Z_{> 0}^2= N \cap \sigma^{\circ}$ be primitive. Then:     
      \[ \trop^{\Delta_{1,1}}(a,b) = \min\{ \mu(f^b - \lambda g^a), \  \lambda \in \C^* \} - (a-1)(b-1) Z(f) \cdot Z(g). \]
\end{proposition}

\begin{proof}
    Applying Proposition \ref{prop:intbin} \eqref{eq:itemeq} to $Z(\eta) = \Delta_{1,1}$, we get:
    \begin{equation} 
       \label{eq:tropdelta1} \trop^{\Delta_{1,1}}(a,b) = \min\{ \Delta_{1,1} \cdot B_{a,b,\lambda}, \  \lambda \in \C^* \}. 
    \end{equation}
    Define $S := \{\lambda \in \C^*, \ f^b - \lambda g^a \mbox{ is not reduced} \}$.
    Theorem \ref{thm:Bthm} applied to $\varphi = \varphi_{1,1}$ and $F = u^b - \lambda v^a$ implies that:
     \begin{equation} 
       \label{eq:intdelta1}   \Delta_{1,1} \cdot B_{a,b,\lambda} = (\mu(f^b - \lambda g^a) -1) - (\mu(u^b - \lambda v^a) -1) \ Z(f) \cdot Z(g) \  \mbox{ for all}  \lambda \in \C^* \setminus S. 
    \end{equation}
    We used the facts that $\varphi^* (u^b - \lambda v^a) = f^b - \lambda g^a$ and $\deg(\varphi_{1,1}) = Z(f) \cdot Z(g)$ (see Proposition \ref{prop:quotform}). 

    As $\lambda \neq 0$, we have $\mu(u^b - \lambda v^a) = (a-1)(b-1)$. Formula \eqref{eq:intdelta1} becomes:
    \begin{equation} 
       \label{eq:intdelta1bis}   \Delta_{1,1} \cdot B_{a,b,\lambda}  = (\mu(f^b - \lambda g^a) -1) - (a-1)(b-1) Z(f) \cdot Z(g) \  \mbox{ for all } \lambda \in \C^* \setminus S. 
    \end{equation}
  Combining the equalities \eqref{eq:tropdelta1} and \eqref{eq:intdelta1bis}, we get the desired statement. 
\end{proof}

Proposition \ref{prop:tropexpr} will be used in the proof of Proposition \ref{prop:combinvjnp} below.

\medskip
\section{Combinatorial invariance of jacobian Newton polygons}
\label{sec:combinvjac}

In this section we prove Theorem \ref{thm:Cthm} of the introduction, which states that the Jacobian Newton polygon $\caln(\Delta_{\varphi})$ of a finite morphism $\varphi = (f,g) : (S,s) \to (\C^2, 0)$ depends only on the combinatorial type of the triple $(S, Z(f), Z(g))$. Our proof is based on Proposition \ref{prop:combinvjnp}, which states in particular that, with a finite number of exceptions, the values $\trop^{\Delta_{1,1}}(a,b)$ of the tropical function of the jacobian Newton polygon of $\varphi$ at primitive vectors $(a,b) \in (\Z_{>0})^2$ depend only on that combinatorial type. 
\medskip

We denote again  by $\boxed{|D|}$ the reduction of an effective divisor $D$ on a smooth complex surface, as in the proof of Theorem \ref{thm:Bthm}. We will look at $|D|$ as a complex curve, speaking therefore about its {\em smooth/singular} points. If $E_j$ is an irreducible component of the exceptional divisor of a good resolution $\pi$ of $(S,s)$ in the sense of \ref{def:termres}, we denote by $\boxed{\ord_{E_j}}$ the corresponding valuation of the local ring $\calo_{S,s}$. That is, $\ord_{E_j}(h)$ denotes the order of vanishing of $\pi^*(h)$ along $E_j$, for every $h \in \calo_{S,s}$.

We will use the following related notions of  {\em combinatorial type}: 

\begin{definition}
    \label{def:combtypefunct}
    Assume that $(S,s)$ is a normal surface singularity and that $f \in \mathfrak{m}_{S,s}$. Let $\pi : (S_{\pi}, E_{\pi}) \to (S, s)$ be a good resolution of $f$. The {\bf combinatorial type of $\pi^*(f)$} is the isomorphism class of the dual graph of the total transform 
    $Z(\pi^*(f))$ of $Z(f)$ on $S_{\pi}$, weighted by the orders of vanishing $\ord_{E_j}(f)$ of $f$ along the irreducible components $E_j$ of $|Z(\pi^*(f))|$ and as usual by the genera and self-intersection numbers of the irreducible components of the exceptional divisor $E_{\pi}$ of $\pi$. 
\end{definition}

\begin{definition}
   \label{def:combtypepair}
    Assume that $(S,s)$ is a normal surface singularity and that $f, g \in \mathfrak{m}_{S,s}$ define a finite morphism $\varphi = (f,g) : (S,s) \to (\C^2, 0)$. The {\bf combinatorial type of the triple $(S, Z(f), Z(g))$} is the combinatorial type of $\pi^*(fg)$ in the sense of Definition \ref{def:combtypefunct}, where $\pi$ is the minimal good resolution of $fg$, in which the vertices corresponding to the strict transforms of the branches of $Z(f)$ and $Z(g)$ are marked by the name of the corresponding germ $Z(f)$ or $Z(g)$. 
\end{definition}

Recall that the {\em pencil $\Pp \varphi : (S,s) \dashrightarrow \C\Pp^1$ generated by $\varphi$} was introduced in Definition \ref{def:specialvalue} and the notion of {\em good resolution which lifts its indeterminacies} in Definition \ref{def:dicrit}. 

\begin{proposition}
  \label{prop:combinvjnp}
     Assume that $f, g \in \mathfrak{m}_{S,s}$ define a finite morphism $\varphi = (f,g) : (S,s) \to (\C^2, 0)$. Let $\pi : (S_{\pi}, E_{\pi}) \to (S, s)$ be a good resolution of $fg$  which lifts the indeterminacies of the pencil $\Pp \varphi$. Denote by $(E_j)_{j \in I(\pi)}$ the irreducible components of the exceptional divisor $E_{\pi}$ of $\pi$. Assume that $(a,b) \in \Z_{> 0}^2$ is primitive and satisfies the hypothesis:
       \begin{equation}
          \label{eq:notaHiroquot}
         \frac{a}{b} \notin \left\{ \frac{\ord_{E_j}(f)}{\ord_{E_j}(g)} , j \in I(\pi) \right\}.  
        \end{equation}
    Then:
      \begin{enumerate}
        \item  \label{eq:indepmu} 
           The Milnor number 
           $\mu(f^b- \lambda g^a)$ is independent of $\lambda \in \C^*$. 
        \item  \label{eq:tropcomb}
           The integer 
          $\trop^{\Delta_{1,1}}(a,b)$ depends only on the combinatorial type of the triple $(S, Z(f), Z(g))$.
        \item  \label{eq:notintrop}
          $(a,b) \notin \trop(\Delta_{1,1})$. 
    \end{enumerate}     
\end{proposition}

\begin{proof}   $\ $

    {\em (1).} The basic idea of the proof is to use A'Campo's formula of Theorem \ref{thm:partACform} in order to compute $\mu(f^b- \lambda g^a)$. One cannot apply the formula  directly on the surface $S_{\pi}$, because $\pi$ is not necessarily a good resolution of $f^b- \lambda g^a$. As a consequence of Claims 1 and 2 below, we will show in Claim 3 that there exists another good resolution of $Z(fg)$ which factors through $\pi$ and which is a common good resolution of all the germs $f^b- \lambda g^a$.  Our proofs of Claims 1 and 2 use arguments which are standard in this context (see for instance \cite[Proof of Theorem 3.1]{G 12} and \cite[Proofs of Lemmas 4.2 and 5.5]{GGP 22}). 

    \medskip
      {\bf Claim 1:} {\em For every $\lambda \in \C^*$, the strict transform of $Z(f^b- \lambda g^a)$ on $S_{\pi}$ passes only through points of $E_{\pi}$ which are singular on the reduced normal crossing divisor $|Z(\pi^*(fg))|$.} 

      \medskip
      {\em Proof.} Let $q \in E_{\pi}$ be a smooth point of $|Z(\pi^*(fg))|$. Choose local coordinates $(x,y)$ on the smooth germ $(S_{\pi}, q)$ such that $|Z(\pi^*(fg))| = Z(y)$ near $q$. Denote by $E_j$ the unique irreducible component of $E_{\pi}$ such that $q \in E_j$. One may write
         $ \pi^*(f) = U_f \cdot y^{\ord_{E_j}(f)}, \ \pi^*(g) = U_g \cdot y^{\ord_{E_j}(g)}$, 
       where $U_f, U_g$ are units of the local ring $\calo_{S_{\pi}, q}$. Therefore:
         \[\pi^*(f^b- \lambda g^a) = U_f^b \cdot y^{b \cdot \ord_{E_j}(f)} - \lambda U_g^a \cdot y^{a \cdot \ord_{E_j}(g)}. \]
       Hypothesis \eqref{eq:notaHiroquot} implies that the exponents $b \cdot \ord_{E_j}(f)$ and $a \cdot \ord_{E_j}(g)$ are distinct, which shows that:
       \[\pi^*(f^b- \lambda g^a) = U \cdot y^{\min\{b \cdot \ord_{E_j}(f), \ a \cdot \ord_{E_j}(g)\}},  \]
       where $U$ is a unit of the local ring $\calo_{S_{\pi}, q}$. Thus the strict transform of $Z(f^b- \lambda g^a)$ on $S_{\pi}$ does not pass through $q$.   

      \medskip
      {\bf Claim 2:} {\em At each singular point $q$ of the reduced curve $|Z(\pi^*(fg))|$, if one chooses local coordinates defining the germ $(|Z(\pi^*(fg))|,q)$, then $\pi^*(f^b- \lambda g^a)$ is Newton non-degenerate in those coordinates, with a Newton polygon which is independent of $\lambda \in \C^*$ and depends only on $a,b$ and on the orders of vanishing of $f$ and $g$ on the irreducible components of $|Z(\pi^*(fg))|$ passing through $q$.}

      \medskip
      {\em Proof.}  The proof is similar to that of Claim 1, but slightly longer. Let $(x,y)$ be local coordinates on the smooth germ $(S_{\pi}, q)$ such that $Z(x)= E_i$ and $Z(y)=E_j$ locally near $q$, where $E_i, E_j$ are the irreducible components of  $|Z(\pi^*(fg))|$ passing through $q$. Note that $E_i$ and $E_j$ may be components of the strict transform of $Z(fg)$. One may write 
         $ \pi^*(f) = U_f \cdot x^{\ord_{E_i}(f)} y^{\ord_{E_j}(f)}, \ \pi^*(g) = U_g \cdot x^{\ord_{E_i}(g)} y^{\ord_{E_j}(g)}$, 
       where $U_f, U_g$ are units of the local ring $\calo_{S_{\pi}, q}$. Therefore:
         \begin{equation} 
            \label{eq:locpb}
           \pi^*(f^b- \lambda g^a) = U_f^b \cdot x^{b \cdot \ord_{E_i}(f)} y^{b \cdot \ord_{E_j}(f)} - \lambda U_g^a \cdot x^{a \cdot \ord_{E_i}(g)} y^{a \cdot \ord_{E_j}(g)}. 
        \end{equation}
       Hypothesis \eqref{eq:notaHiroquot} shows that
         $b \cdot \ord_{E_i}(f) \neq a \cdot \ord_{E_i}(g)$ and $b \cdot \ord_{E_j}(f) \neq a \cdot \ord_{E_j}(g)$. 
    If 
        \begin{equation} 
           \label{eq:sametypeineq}
          b \cdot \ord_{E_i}(f) \bowtie  a \cdot \ord_{E_i}(g) \ \ \mathrm{and} \ \ b \cdot \ord_{E_j}(f) \bowtie a \cdot \ord_{E_j}(g), \ \mathrm{where} \  \bowtie \ \in \{<, >\}, 
        \end{equation}
    then the equality \eqref{eq:locpb} implies that:
      \[\pi^*(f^b- \lambda g^a) = U \cdot x^{\min\{b \cdot \ord_{E_i}(f), \ a \cdot \ord_{E_i}(g)\}} y^{\min\{b \cdot \ord_{E_j}(f), \ a \cdot \ord_{E_j}(g)\}},  \]
       where $U$ is a unit of the local ring $\calo_{S_{\pi}, q}$. This shows that the strict transform of $Z(f^b- \lambda g^a)$ on $S_{\pi}$ does not pass through $q$. Moreover, with the notations of the beginning of Section \ref{sec:Npolint}, the Newton polygon of $\pi^*(f^b- \lambda g^a)$ in the coordinates $(x,y)$ is the translated non-negative quadrant
        \[ \sigma^{\vee} + (\min\{b \cdot \ord_{E_i}(f), \ a \cdot \ord_{E_i}(g)\} , \min\{b \cdot \ord_{E_j}(f), \ a \cdot \ord_{E_j}(g)\}), \]
       which depends only on $a,b$ and 
       on the orders of vanishing of $f$,  $g$ on $E_i, E_j$. 

       Assume now that \eqref{eq:sametypeineq} is false. Suppose for instance that 
       \begin{equation} 
           \label{eq:revineq}
          b \cdot \ord_{E_i}(f) <  a \cdot \ord_{E_i}(g) \ \mathrm{and} \  b \cdot \ord_{E_j}(f) >  a \cdot \ord_{E_j}(g), 
        \end{equation}
    the reverse pair of inequalities  having a similar treatment. Relations \eqref{eq:locpb} and \eqref{eq:revineq} imply that
    \begin{equation} 
      \label{eq:factorpb}
           \pi^*(f^b- \lambda g^a) = x^{b \cdot \ord_{E_i}(f)} y^{a \cdot \ord_{E_j}(g)}\left( U_f^b \cdot  y^{b \cdot \ord_{E_j}(f) - a \cdot \ord_{E_j}(g)} - \lambda U_g^a \cdot x^{a \cdot \ord_{E_i}(g) - b \cdot \ord_{E_i}(f)} \right). 
    \end{equation}

\begin{figure}[ht!] 
\begin{center}
\begin{tikzpicture}[x=0.55cm,y=0.55cm] 
% [x=0.6cm,y=0.6cm]
\tikzstyle{every node}=[font=\small]
\fill[fill=yellow!40!white] (0,6) --(0,5)-- (10,0) -- (12, 0) -- (12,6) --cycle;

\draw[->] (0,0) -- (12,0) 
node[right] {$p$};
\draw[->] (0,0) -- (0,6) node[above] {$q$};

\draw [ultra thick, color=black](0,5) -- (0,6);
\draw [ultra thick, color=blue](10,0) -- (0,5);
\draw [ultra thick, color=black](10,0) -- (12,0);

\node[draw,circle,inner sep=1.5pt,fill=violet!40, color=violet!40] at (10,0) {};
%\node[draw,circle,inner sep=1.5pt,fill=orange!50,color=orange!50] at (3,3) {};
\node[draw,circle,inner sep=1.5pt,fill=orange, color=orange] at (0,5) {};

\node [below] at (8.5,-0.2) {$a \cdot \ord_{E_i}(g) - b \cdot \ord_{E_i}(f)$} ;
\node [left] at (0,5) {$a \cdot \ord_{E_j}(f) - b \cdot \ord_{E_j}(g)$} ;

 \end{tikzpicture}
\end{center}
\caption{The Newton polygon of the strict transform of $Z(f^b - \lambda g^a)$ when one has relation \eqref{eq:factorpb}.}  
   \label{fig:NPoppineq}
    \end{figure}
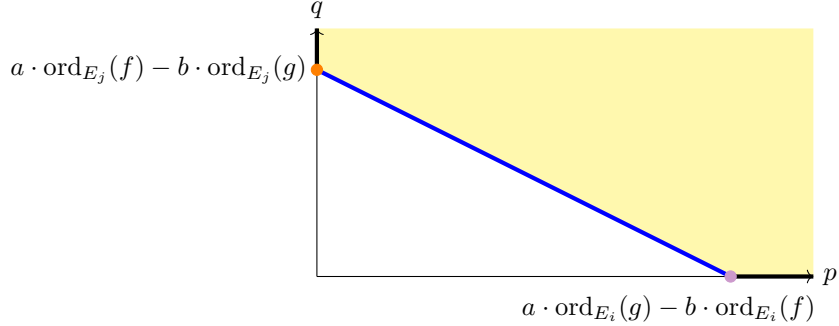

The factor $h$ of the right side of \eqref{eq:factorpb} which is written between brackets defines the strict transform of $Z(f^b - \lambda g^a)$ by $\pi$ in the neighborhood of $q$. Its Newton polygon $\caln(h)$ is shown in Figure \ref{fig:NPoppineq}. The restriction $h_{\call}$  of $h$ to the unique compact edge $\call$ of $\caln(h)$ (in the sense of Definition \ref{def:initNP}) is:
 \[h_{\call} = U_f(0,0)^b \cdot  y^{b \cdot \ord_{E_j}(f) - a \cdot \ord_{E_j}(g)} - \lambda U_g(0,0)^a \cdot x^{a \cdot \ord_{E_i}(g) - b \cdot \ord_{E_i}(f)}. \]
As $U_f, U_g$ are units of $\calo_{S_{\pi}, q}$, we have  $U_f(0,0) \neq 0\neq U_g(0,0)$. Moreover, $\lambda \neq 0$ by hypothesis.  Therefore, the univariate polynomial $\calp_{h, \call}(z)$ in the sense of Definition \ref{def:asspol} is a binomial. Because univariate binomials have only simple roots, we deduce from Proposition \ref{prop:Nndsimpleroots} that $h$ is Newton non-degenerate in the coordinates $(x,y)$. Therefore, equation \eqref{eq:factorpb} shows that $\pi^*(f^b- \lambda g^a)$ is also Newton non-degenerate, and that its associated Newton polygon is the translation by the vector $(b \cdot \ord_{E_i}(f),  a \cdot \ord_{E_j}(g))$ of the polygon of Figure \ref{fig:NPoppineq}. As a consequence, Claim 2 is proved also in this case.

\medskip
      {\bf Claim 3:} {\em There exists a good resolution $\psi : (S_{\psi}, E_{\psi}) \to (S, s)$ of $fg$ obtained by composing $\pi$ with a sequence of blow ups determined by $a,b$ and the combinatorial types of $\pi^*(f) $ and $\pi^*(g)$, which is a common good resolution of all the functions $f^b- \lambda g^a$ for varying $\lambda \in \C^*$}.

    \medskip
      {\em Proof.} Claim 1 implies that one may get a good resolution of $f^b- \lambda g^a$ by blowing up only singular points of $|Z(\pi^*(fg))|$ and points infinitely near them. Claim 2 implies that the compositions of those blow ups may be chosen to be toric in the local coordinates defining locally $|Z(\pi^*(fg))|$ near each singular point $q$ of this reduced divisor and that the corresponding toric morphism is determined by the common Newton polygon of all the germs $\pi^*(f^b- \lambda g^a)$ (see \cite[Proposition 1.4.20]{GBGPPP 20}).  By \cite[Paragraph before Definition 1.3.33]{GBGPPP 20}, that toric morphism is independent of the local coordinates defining $|Z(\pi^*(fg))|$. We get in this way a good resolution $\psi : (S_{\psi}, E_{\psi}) \to (S, s)$ of all the functions $f^b- \lambda g^a$,  factoring moreover through $\pi$. It is obtained by performing blow ups determined by the combinatorial types of $\pi^*(f) $ and $\pi^*(g)$. 

      \medskip
      Applying now A'Campo's formula (see Theorem \ref{thm:partACform}) to the good resolution $\psi$ of Claim 3, we see that the Milnor number 
      $\mu(f^b- \lambda g^a)$ is indeed independent of $\lambda \in \C^*$.

      \medskip
    {\em (2).}  By Proposition \ref{prop:tropexpr}, it is enough to show that $\min\{ \mu(f^b - \lambda g^a), \  \lambda \in \C^* \}$ depends only on the combinatorial type of the triple $(S, Z(f), Z(g))$. By point {\em (1)}, this minimum is equal to 
    $\mu(f^b - \lambda g^a)$ for every 
    $\lambda \in \C^*$. By Claim 2 above, the Newton polygon of $\pi^*(f^b- \lambda g^a)$ at each intersection point of two components $E_i, E_j$ of $|Z(\pi^*(fg))|$ is determined by the positive integers $a, b, \ord_{E_i}(f), \ord_{E_i}(g), \ord_{E_j}(f), \ord_{E_j}(g)$. Therefore, we get a good resolution of all the curves $Z(f^b- \lambda g^a)$ by performing blow ups above the singular points of $|Z(\pi^*(fg))|$ which are determined by the pair $(a,b)$ and the orders of vanishing 
    of $f$ and $g$ along the irreducible components of $E_{\pi}$. Applying this fact to the minimal good resolution of $fg$, we get from Definition \ref{def:combtypepair} that $\mu(f^b - \lambda g^a)$ has a good resolution $\psi$ whose combinatorial type is determined by the combinatorial type of $(S, Z(f), Z(g))$. Applying A'Campo's formula of Theorem \ref{thm:partACform} to the germ $f^b - \lambda g^a$ and the good resolution $\psi$, one gets the desired conclusion. 

    \medskip
   {\em (3).} Assume by contradiction that $(a,b) \in \trop(\Delta_{1,1})$. By Definitions \ref{def:loctrop} and \ref{def:invobj} of the local tropicalisation,  this means that the support face $\call(a,b)$ of the vector $(a,b)$ in the Newton polygon $\caln(\Delta_{1,1})$ (see Definition \ref{def:suppface}) is a compact edge of $\caln(\Delta_{1,1})$. Therefore, the corresponding univariate polynomial $\calp_{\Delta_{1,1},\call}$ of Definition \ref{def:invobj} has at least one root $\lambda_0 \in \C^*$. By Proposition \ref{prop:equivdist}, we deduce that $[\lambda_0:1]$ is a $\varphi_{a,b}$-distinguished point of $\C\Pp^1$ in the sense of Definition \ref{def:specialvalue}. That is:
    \[ \Delta_{a,b} \cdot B_{1,1, \lambda_0}  > \min \{ \Delta_{a,b} \cdot B_{1,1, \lambda}, \  
              \lambda \in \C^* \}.\]
   Corollary \ref{cor:equiv} implies that: 
   \[ \Delta_{1,1} \cdot B_{a,b, \lambda_0}  > \min \{ \Delta_{1,1} \cdot B_{a,b, \lambda} , \  
              \lambda \in \C^* \}.\]
    Formula \eqref{eq:intdelta1bis} (which is a consequence of the generalised L\^e and Casas-Alvero formula of Theorem \ref{thm:Bthm}) shows therefore that:
    \[ \mu(f^b- \lambda_0 g^a) > \mu(f^b- \lambda g^a).\]
    This contradicts point  \eqref{eq:indepmu}.
\end{proof}

\begin{remark}
   As a consequence of Proposition \ref{prop:combinvjnp} \eqref{eq:notintrop}, the set of slopes of the rays forming the local tropicalization $\trop(\Delta_{1,1})$ of the discriminant divisor of $\varphi$ is included in the set $\left\{ \frac{\ord_{E_j}(f)}{\ord_{E_j}(g)} , j \in I(\pi) \right\}$. This result is also a consequence of L\^e, Maugendre and Weber's 
    \cite[Corollary 3.4]{LMW 01}, proved by completely different methods, in which those slopes are called {\em discriminantal ratios} and the quotients $\frac{\ord_{E_j}(f)}{\ord_{E_j}(g)}$ are called {\em Hironaka numbers}. L\^e, Michel and Weber already wrote in \cite{LMW 89} about {\em coefficient de contact (au sens de Hironaka)}, giving as a reference Hironaka's \cite[Page 5]{H 74} (see also \cite[Definition 3.1.2]{AHV 18}).  In \cite{LMW 01}  L\^e, Maugendre and Weber describe a set of components $E_j$ of $E_{\pi}$ whose associated Hironaka numbers form the whole set of discriminantal ratios (see \cite[Theorem 0.3]{LMW 01}). This generalizes to arbitrary normal surface singularities Maugendre's result \cite[Theorem 1]{M 99}, which concerned {\em smooth} germs.  
\end{remark}

\begin{remark} \label{rem:reasonab}
    The reason why we consider throughout the paper pairs $(f^b, g^a)$ instead of pairs $(f^a, g^b)$ is that we want to write condition \eqref{eq:notaHiroquot} with the ratios $a/b$ instead of $b/a$ to be compared with the ratios $\frac{\ord_{E_j}(f)}{\ord_{E_j}(g)}$. It is in this sense that $a$ corresponds to $f$ and $b$ to $g$. 
\end{remark}

Proposition \ref{prop:combinvjnp} implies the following generalization of Gwo\'zdziewicz' \cite[Theorem 2.1]{G 12} from smooth sources $(S,s)$ to arbitrary normal sources: 

\begin{Ctheorem}
   \label{thm:Cthm}
     Let $(S,s)$ be a normal surface singularity. 
    Assume that $f, g \in \calo_{S,s}$ define a finite morphism $\varphi = (f,g) : (S,s) \to (\C^2, 0)$. Then the Jacobian Newton polygon $\caln(\Delta_{\varphi})$ of $\varphi$ depends only on the combinatorial type of $(S, Z(f), Z(g))$.
\end{Ctheorem}

\begin{proof}
    By Proposition \ref{prop:reconstr}, $\caln(\Delta_{\varphi})$ is determined by the tropical function $\trop^{\Delta_{\varphi}} : \sigma \to \R_{\geq 0}$. Proposition \ref{prop:combinvjnp} shows that $\trop^{\Delta_{\varphi}}(a,b)$ depends only on the combinatorial type of $(S, Z(f), Z(g))$,  for all primitive vectors $(a,b) \in (\Z_{>0})^2$ outside a finite set. As $\trop^{\Delta_{\varphi}}$ is {\em homogeneous of degree one}, its restriction to the union of all rational rays but a finite number of them has the same property. As this function is moreover {\em continuous} and the previous union is dense in its domain $\sigma$, we deduce that it is determined by the combinatorial type of $(S, Z(f), Z(g))$. Therefore, the same is true for $\caln(\Delta_{\varphi})$.
\end{proof}

As explained in the introduction, Theorem \ref{thm:Cthm} may be also obtained as a consequence of Michel's \cite[Theorems 4.8, 4.9]{M 08}.

\medskip
\section{Comparison of divisorial and Fitting critical loci and discriminants}
\label{sec:compdivFitt}

In this subsection we prove results which are used in the proof of Proposition 
\ref{prop:holfamdiv}, used in turn in the proof of Proposition \ref{prop:samepol}. 
Proposition \ref{prop:holfamdiv} states that the discriminant divisors of a special 
one-parameter family of finite morphisms vary holomorphically. In order to get 
this holomorphic dependency, we need to prove that discriminant divisors defined 
in arbitrary dimension commute with base change. 
This property was proved by Teissier \cite[pages 341-342]{T 73} 
for another notion of discriminant spaces, defined using {\em Fitting ideals} 
(see Proposition \ref{prop:compatdiscrbc}). 
We show that we may apply it to our definition of discriminant divisors 
by proving that the two notions coincide for finite morphisms between {\em smooth} 
varieties of the same dimension (see Proposition \ref{prop:equalspaces}). 
\medskip

If $X$ is a complex space, we denote by $\boxed{\calo_X}$ its structure sheaf. 
If $\calm$ is a coherent sheaf of $\calo_X$-modules, we simply say 
that  $\calm$ is an {\bf $\calo_X$-module}. Similarly, we will speak about {\bf $\calo_X$-ideals}. 
We say that the complex subspace defined by the $\calo_X$-ideal $I$ is its {\bf zero locus} 
and we denote it by $\boxed{Z(I)} \hookrightarrow X$. Let $\boxed{|Z(I)|}$ be its underlying subset. 
If $I$ is a principal $\calo_X$-ideal, 
we will also think about $\boxed{Z(I)}$ as an effective Cartier divisor on $X$.

We will be interested in the {\em Fitting $\calo_X$-ideals of $\calo_X$-modules}. 
They are defined sheaf-theoretically using the analogous notion for modules over rings:

\begin{definition}   \label{def:Fitt}
    Let $R$ be a commutative ring with $1$ and $M$ be an $R$-module of finite presentation. Let 
      \[ R^q  \xrightarrow{\psi}  R^p \xrightarrow{}  M  \xrightarrow{}  0 \]
    be such a presentation of $M$. Let us identify the morphism $\psi$ of $R$-modules 
    with its matrix in the canonical bases of $R^q$ and  $R^p$. 
    For every $j \in \Z$, the {\bf $j$-th Fitting ideal of $M$} is the ideal 
    of the ring $R$ defined by:
       \[ \boxed{F_j(M)} := \left\{ \begin{array}{ll}
                                       0      &  \mbox{if }  p - j >   \min\{p, q\}  ,  \\
                                       \mbox{the ideal generated by the } p-j \mbox{ minors of } \psi  &  
                                            \mbox{if }  0 \leq p - j   \leq \min\{p, q\} ,   \\
                                       R   &  \mbox{if }  p - j  < 0.
                                   \end{array} \right.  \]
\end{definition} 

As proved in \cite[Appendix D]{K 86}, this definition is independent of the choice 
of presentation of $M$. One has obviously:
   \[  0 = F_{-1}(M) \subseteq 
          F_0(M) \subseteq F_1(M) \subseteq F_2(M) \subseteq \cdots  \subseteq F_p(M) = R \]
 The presheaf $F_j(\calm)$ of $j$-th Fitting ideals of a sheaf of $\calo_X$-modules $\calm$ 
 is a sheaf.  Its zero-locus is the set of points of $X$ where the fiber of $\calm$ 
 has rank at least $j + 1$. 

In the sequel, all varieties will be considered to be {\em pure dimensional}. 
We will consider a finite morphism  $\pi : X \to Y$ between complex analytic varieties 
 of the same dimension.  This is a special case 
of the situation considered by Teissier in \cite[Sections 1 and 2]{T 77}, where arbitrary finite 
morphisms between complex analytic spaces are considered. 

Point (\ref{FDI}) of the following definition is a special case of Teissier's definition 
of \cite[page 572]{T 77} and point (\ref{DDI})  is an analog in arbitrary dimension of 
Definition \ref{def:dirimdiv}:

\begin{definition}   \label{def:twodirim}
    Let $\pi : X \to Y$ be a finite morphism between complex analytic varieties 
    of the same dimension. Let $A$ be an effective Weil divisor of $X$. 
      \begin{enumerate}
          \item \label{FDI}
             The {\bf Fitting direct image $\boxed{\pi_*^F A}$ of $A$ by $\pi$} is the subspace 
          $Z(F_0(\pi_* \calo_A))$ of $Y$. 
             \item \label{DDI}
             Let $A_i$ be the irreducible components of $A$. Consider the effective divisors 
      \[\pi_*^D A_i := \deg(\pi|_{A_i}) \ \pi(A_i) \]
   of $Y$. Then, the {\bf divisorial direct image $\boxed{\pi_*^D A}$ of $A$ by $\pi$} 
    is then defined by linearity. 
      \end{enumerate}
\end{definition}

Consider now the module $\boxed{\Omega^1_{\pi}}$ of {\bf holomorphic $1$-forms on $X$ 
  relative to the 
morphism $\pi$} (see \cite[Section 2.2]{T 77}, where it is denoted $\Omega^1_{X | Y}$). There exists a canonical exact sequence 
of $\calo_X$-modules:
  \begin{equation} \label{eq:canexactdiff}
      \pi^* \Omega^1_Y \to \Omega^1_X \to \Omega^1_{\pi} \to 0. 
  \end{equation}
  
 Definitions \ref{def:twocritloci} (\ref{DCS}) and \ref{def:twodiscrloci} (\ref{DDS})  
 below of {\em divisorial} critical and discriminant subspaces are analogs of 
 Definition \ref{def:critdisc}. Note that  for both 
 of them, the spaces $X$ and $Y$ are assumed to be {\em smooth}.

\begin{definition}   \label{def:twocritloci}
    Let $\pi : X \to Y$ be a finite morphism between complex analytic varieties 
    of the same dimension. 
      \begin{enumerate}
             \item \label{FCS}
             The {\bf Fitting critical subspace $\boxed{\Xi^F_{\pi}}$ of $\pi$} is the subspace 
          $Z(F_0(\Omega^1_{\pi}))$ of $X$. 
          \item \label{DCS} 
               Assume that $X$ and $Y$ are {\em smooth}. 
               The {\bf divisorial critical subspace $\boxed{\Xi^D_{\pi}}$ of $\pi$} is the effective divisor 
               $Z(\pi^* \omega_Y)$, where $\omega_Y$ is any holomorphic nowhere zero form of 
               maximal degree defined on a non-empty open subset of $Y$.
      \end{enumerate}
\end{definition}

A holomorphic nowhere zero form of maximal degree may not exist globally on $Y$. This poses no problem, because two such forms defined on the same open subset of $Y$ lead to the same divisorial critical subspace on $X$. One may therefore work with a covering of $Y$ by open subsets on which such forms exist. 

The following definition is obtained by combining Definitions \ref{def:twodirim} and \ref{def:twocritloci}:

\begin{definition}   \label{def:twodiscrloci}
    Let $\pi : X \to Y$ be a finite morphism between complex analytic varieties 
    of the same dimension. 
      \begin{enumerate}
             \item \label{FDS}
             The {\bf Fitting discriminant subspace $\boxed{\Delta^F_{\pi}}$ of $\pi$} 
               is the effective divisor $\pi^F_*  \Xi^F_{\pi}$ of $Y$.
          \item \label{DDS}
                Assume that $X$ and $Y$ are {\em smooth}. 
               The {\bf divisorial discriminant subspace $\boxed{\Delta^D_{\pi}}$ of $\pi$} 
               is the effective divisor $\pi^D_*  \Xi^D_{\pi}$ of $Y$.
      \end{enumerate}
\end{definition}

In the following statement, we assume from the start that 
both $X$ and $Y$ are  {\em smooth}: 

\begin{proposition}   \label{prop:equalspaces}
    Let $\pi : X \to Y$ be a finite morphism between {\em smooth} complex analytic varieties 
    of the same  dimension. 
        \begin{enumerate}
            \item   \label{equalDI}
                If $A$ is an effective divisor of $X$, then we have $\pi^F_* A  = \pi^D_* A$. 
            \item   \label{equalCS}
               The Fitting and divisorial critical subspaces of $X$ coincide: 
                   $\Xi^F_{\pi} = \Xi^D_{\pi}$. 
            \item \label{equalDS}
               The Fitting and divisorial discriminant subspaces of $Y$ coincide:  
                   $\Delta^F_{\pi} = \Delta^D_{\pi}$. 
        \end{enumerate}
\end{proposition}

\begin{proof}   $  \  $

   $\bullet$ {\bf Proof of (\ref{equalDI}).} 
    Localizing in an open neighborhood of any point of $Y$ and on any connected component 
    of its preimage by $\pi$, we may assume that $X$ and $Y$ are connected and that $A$ is defined 
    by a global holomorphic function $\alpha : X \to \C$. Consider the standard presentation of 
    the $\calo_X$-module $\calo_A$, where the first morphism is the multiplication by $\alpha$:
        \[  \calo_X \xrightarrow{\cdot \alpha}  \calo_X \to \calo_A. \]
     Since the direct images of exact sequences of sheaves of abelian groups under finite morphisms 
     remain exact (see \cite[Chapter I.3, Theorem 4]{GR 79}), we get the following exact 
     sequence of $\calo_Y$-modules:
         \begin{equation} \label{direxactseq}
              \pi_*  \calo_X \xrightarrow{\cdot \alpha}  \pi_*  \calo_X \to   \pi_*  \calo_A. 
         \end{equation}
      As $X$ and $Y$ are smooth of the same dimension and $\pi$ is a finite morphism, the 
      pullback by $\pi$ of any local coordinate system on the target $Y$ is a regular sequence on $X$. 
      By \cite[Corollary on page 154]{F 76}, we deduce that $\pi$ is a {\em flat} morphism. 
      Using again the finiteness of $\pi$, it follows from \cite[Proposition 3.13]{F 76} that 
      $\pi_* \calo_X$ is a locally free $\calo_Y$-module. 
      
      After possibly shrinking again $Y$ and $X$, we may assume therefore that $\pi_* \calo_X$ 
      is free on $Y$. That is, there exists an isomorphism $\pi_* \calo_X \simeq \calo^n_Y$ of 
      $\calo_Y$-modules, where $n > 0$ is the degree of $\pi$. The exact sequence 
      (\ref{direxactseq}) becomes therefore:
          \[  \calo^n_Y \xrightarrow{\mu_{\alpha}}  \calo^n_Y \to   \pi_*  \calo_A,  \]
        where $\mu_{\alpha}$ is the endomorphism corresponding to the multiplication by 
        $\alpha$ on $\calo_X$. Definition \ref{def:Fitt} yields:
            \begin{equation}   \label{eq:Fittdet}
                  F_0(\pi_*  \calo_A) = ( \det(\mu_{\alpha})), 
            \end{equation}
        that is, the $0$-th Fitting $\calo_Y$-ideal $F_0(\pi_*  \calo_A)$ is principal, generated 
        by the determinant $ \det(\mu_{\alpha})$. But this determinant is the {\em norm} 
        $N(\alpha)$ of $\alpha$ relative to the finite morphism $\pi$. 
        By \cite[Remark 2.19 of page 272]{L 02}, the divisorial direct image $\pi^D_* A $ 
        is defined by $N(\alpha)$, that is, by $\det(\mu_{\alpha})$. From relation 
        (\ref{eq:Fittdet}) we deduce that:
           \[ \pi_*^F A = Z(F_0(\pi_* \calo_A)) = Z(\det(\mu_{\alpha})  ) =   \pi^D_* A. \]

           \medskip
            $\bullet$ {\bf Proof of (\ref{equalCS}).}  Consider the exact sequence (\ref{eq:canexactdiff}) 
               of modules  of differentials. After 
               shrinking $Y$ and $X$, we may assume that both are embeddable 
               as open subsets of $\C^m$, where $m > 0$ is their common dimension. 
               Let $\underline{x} := (x_1, \dots, x_m)$ and $\underline{y} := (y_1, \dots, y_m)$ be global 
               coordinates on $X$ and $Y$ respectively. Then:
                \[ \left\{   \begin{array}{l} 
                                 \pi^* \Omega^1_Y = \displaystyle{\bigoplus_{i = 1}^m} \  \calo_X \pi^*(dy_i), \\
                                 \Omega^1_X  = \displaystyle{\bigoplus_{i = 1}^m} \ \calo_X dx_i, 
                               \end{array}  \right.   \]
              and the exact sequence (\ref{eq:canexactdiff})  becomes:
                  \[ \bigoplus_{i = 1}^m \calo_X \pi^*(dy_i) \hookrightarrow
                                \bigoplus_{i = 1}^m \calo_X dx_i \to \Omega^1_{\pi} \to 0. \]
              The matrix of the left morphism of $\calo_X$-modules is 
              the Jacobian matrix $\mbox{Jac}_{\underline{x}}(\pi^* \underline{y})$ of $\pi$ 
              relative to the coordinate systems $\underline{x}$ and $\underline{y}$. Therefore:
                  \begin{equation}  \label{eq:FittJac}
                      F_0(\Omega^1_{\pi}) = (\mbox{Jac}_{\underline{x}}(\pi^* \underline{y})). 
                  \end{equation}
                             
       Consider now the nowhere zero holomorphic form 
        $\omega_Y := dy_1 \wedge \cdots \wedge dy_m$ of maximal degree on $Y$. Then:
               \[  \pi^* \omega_Y = \pi^* dy_1 \wedge \cdots \wedge \pi^* dy_m = 
                    \mbox{Jac}_{\underline{x}}(\pi^* \underline{y})  \  dx_1 \wedge \cdots dx_m, \]   
        which yields:
          \[   Z( \pi^* \omega_Y ) =  Z( \mbox{Jac}_{\underline{x}}(\pi^* \underline{y})) .  \]
        Combining this equality with equality (\ref{eq:FittJac}), we get:
           \[    \Xi^F_{\pi} = Z(F_0(\Omega^1_{\pi}) ) =  
                   Z( \mbox{Jac}_{\underline{x}}(\pi^* \underline{y})) =  Z( \pi^* \omega_Y ) 
                   =  \Xi^D_{\pi}. \]
         The first and last equalities follow from Definition \ref{def:twocritloci}. 
        
        \medskip
            $\bullet$ {\bf Proof of (\ref{equalDS}).}  It is enough to combine the properties 
                of points (\ref{equalDI}) and (\ref{equalDS}): 
                 \[ \Delta^F_{\pi} = \pi_*^F \Xi^F_{\pi} =  \pi_*^F \Xi^D_{\pi} = \pi_*^D \Xi^D_{\pi} =   
                      \Delta^D_{\pi}. \]                           
\end{proof}

Proposition \ref{prop:compatdiscrbc} below was stated by Teissier in \cite[pages 341-342]{T 73}. 
He gave the following sketch of proof: ``{\em Indeed, the formation 
of the critical subspace is [compatible with base change], as well as that of [direct image] and that 
of the Fitting ideal}''. Here we give a more detailed 
and formalized proof.

\begin{proposition}   \label{prop:compatdiscrbc}
    The formation of Fitting discriminant spaces is compatible with base changes. 
\end{proposition}

\begin{proof}  Consider a finite morphism $\pi : X \to Y$ satisfying the hypotheses 
  of Definition \ref{def:twodiscrloci}. Denote by $h_Y : \tilde{Y} \to Y$ an arbitrary morphism 
  of base change. It yields the following cartesian diagram:
   \begin{equation} \label{eq:firstdiag}
          \begin{tikzcd}
               \tilde{X}   \arrow[r, "h_X"] \arrow[d,  "\tilde{\pi}" ']
                     \arrow[dr, phantom, "\lrcorner", very near end]
                    &  X \arrow[d, "\pi"] \\
                \tilde{Y}  \arrow[r, "h_Y" ']    &   Y
            \end{tikzcd}   
       \end{equation}
     Finiteness being preserved under base change, the morphism $\tilde{\pi}$ is also finite. 
     As the formation of modules of relative $1$-forms is compatible with 
     base changes (see \cite[page 581]{T 77}), one has 
     $\Omega^1_{\tilde{\pi}} = h_X^* \Omega^1_{\pi}$. 
     Therefore, using the fact that the formation of Fitting ideals is also compatible 
     with base changes (see \cite[page 570]{T 77} or \cite[Proposition D.4]{K 86}),  we get:
     $ F_0(\Omega^1_{\tilde{\pi}})  = F_0( h_X^* \Omega^1_{\pi}) = h_X^* F_0(\Omega^1_{\pi})$. 
     By Definition \ref{def:twocritloci} (\ref{FCS}), we deduce that 
     $ \Xi^F_{\tilde{\pi}} = h_X^* \Xi^F_{\pi}$.  
     This equality means that the following diagram, whose vertical arrows 
     are the natural embeddings,  is cartesian: 
           \begin{equation} \label{eq:secdiag}
                \begin{tikzcd}
                    \Xi^F_{\tilde{\pi}}   \arrow[r, "h_X | _{\Xi^F_{\tilde{\pi}}}"] \arrow[d]
                                 \arrow[dr, phantom, "\lrcorner", very near end]
                               &  \Xi^F_{\pi} \arrow[d] \\
                    \tilde{X}  \arrow[r, "h_X" ']    &   X
              \end{tikzcd}   
           \end{equation}
      As the composition of two cartesian diagrams is again cartesian, composing (\ref{eq:firstdiag}) 
      and (\ref{eq:secdiag}) we get the cartesian diagram:
          \[   \begin{tikzcd}
                    \Xi^F_{\tilde{\pi}}   \arrow[r, "h_X  | _{\Xi^F_{\tilde{\pi}}}"] 
                              \arrow[d,  "\tilde{\pi} | _{\Xi^F_{\tilde{\pi}} }" ']
                                 \arrow[dr, phantom, "\lrcorner", very near end]
                          &  \Xi^F_{\pi} \arrow[d, "\pi | _{\Xi^F_{\pi}}"] \\
                    \tilde{Y}  \arrow[r, "h_Y" ']    &   Y
              \end{tikzcd}   \]
        The formation of direct images of structure sheaves being compatible with base changes 
        (see \cite[page 571]{T 77}), we deduce that 
          $\tilde{\pi}_* \calo_{\Xi^F_{\tilde{\pi}}}  = h_Y^* \pi_* \calo_{\Xi^F_{\pi}}$. Therefore 
        $F_0(\tilde{\pi}_* \calo_{\Xi^F_{\tilde{\pi}}} )  = F_0(h_Y^* \pi_* \calo_{\Xi^F_{\pi}})  
        = h_Y^* F_0(\pi_* \calo_{\Xi^F_{\pi}})$. For the last equality we have used again the 
        compatibility of Fitting ideals with base changes. As a consequence:
            \[   \Delta^F_{\tilde{\pi}} =   Z( F_0(\tilde{\pi}_* \calo_{\Xi^F_{\tilde{\pi}}} ) )=   
                    h_Y^*  Z( F_0(\pi_* \calo_{\Xi^F_{\pi}} ) ) =  h_Y^* \Delta^F_{\pi},  \]
        where the first and last equalities follow from Definition \ref{def:twodiscrloci} (\ref{FDS}).     
        The resulting equality $\Delta^F_{\tilde{\pi}} = h_Y^* \Delta^F_{\pi}$ proves the claim. 
\end{proof}

\medskip
\section{Invariance of the initial Newton curve of a finite morphism}
\label{sec:invINC}

In this section we prove Theorem \ref{thm:Athm} of the introduction, which states that the initial Newton curve of the discriminant divisor of a finite morphism $(f,g) : (S,s) \to (\C^2, 0)$ depends only on the principal divisors $Z(f), Z(g)$ defined by $f$ and $g$. As in \cite{GGP 22}, we deduce this fact from a preliminary result which states that if $\alpha, \beta \in \mathfrak{m}_{S,s}$, then the initial Newton curves of the discriminant divisors of the finite morphisms $(f,g)$ and $(\tilde{f}, \tilde{g}) := ((1 +  \alpha) f, (1 +  \beta) g)$  coincide (see Proposition \ref{prop:mainprep}). In turn, we prove this statement by looking at the univariate polynomials $\calp_{\Delta, \call}$ and $\calp_{\tilde{\Delta}, \call}$ associated to the compact edges $\call$ of the Jacobian Newton polygon $\caln(\Delta)$  of $(f,g)$, which is equal by Theorem \ref{thm:Cthm} to the Jacobian Newton polygon $\caln(\tilde{\Delta})$  of $(\tilde{f}, \tilde{g})$. We prove first that $\calp_{\Delta, \call}$ and 
$\calp_{\tilde{\Delta}, \call}$  have the same set of roots (see Proposition \ref{prop:invroots}), 
then that the corresponding multiplicities also coincide (see Proposition \ref{prop:samepol}). 
We deduce this constancy of multiplicities from the result that the discriminants 
of a certain one-parameter family of finite morphisms 
depend holomorphically of the parameter (see Proposition \ref{prop:holfamdiv}).
\medskip

Let $f, g, \alpha, \beta \in \mathfrak{m}_{S,s}$. Define:
   \begin{equation}
      \label{eq:tildefunct} 
       \boxed{\tilde{f}} := (1 +  \alpha) f, \ \ 
       \boxed{\tilde{g}} := (1 +  \beta) g.
   \end{equation}
Therefore $Z(f) = Z(\tilde{f})$ and $Z(g) = Z(\tilde{g})$. 
We assume that $\varphi = (f,g) : (S,s) \to (\C^2, 0)$ is finite. By Proposition \ref{prop:finchar},  $\boxed{\tilde{\varphi}} = (\tilde{f}, \tilde{g}) : (S,s) \to (\C^2, 0)$ is also finite. Denote by $\boxed{\Delta}$ and $\boxed{\tilde{\Delta}}$ their respective discriminant curves. By Theorem \ref{thm:Cthm}, the Jacobian Newton polygons $\caln(\Delta)$ and $\caln(\tilde{\Delta})$ of the morphisms $\varphi$ and $\tilde{\varphi}$ coincide. Consider a compact edge $\call$ of this polygon and the associated univariate polynomials $\calp_{\Delta, \call}$ and $\calp_{\tilde{\Delta}, \call}$, in the sense of Definition \ref{def:invobj}.

\begin{proposition}
    \label{prop:invroots}
    The polynomials $\calp_{\Delta, \call}$ and $\calp_{\tilde{\Delta}, \call}$ have the same sets of roots in $\C^*$. 
\end{proposition}
    
\begin{proof}
    Let $(a,b)$ be the unique primitive vector of $\Z_{>0}^2$ which is orthogonal to $\call$. Define the finite morphism $\boxed{\tilde{\varphi}_{a,b}} := (\tilde{f}^{b}, \tilde{g}^{a})$ similarly to the definition of the finite morphism $\varphi_{a,b} := (f^b, g^a)$ studied in Section \ref{sec:discfam}.   
    By Proposition \ref{prop:equivdist} combined with Delgado and Maugendre's Theorem \ref{thm:DelMau}, $\lambda \in \C^*$ is a root of the polynomial $\calp_{\Delta, \call}$ if and only if $[\lambda : 1]$ is a $\varphi_{a,b}$-special point of $\C\Pp^1$ in the sense of Definition \ref{def:speczones}. Similarly, $\lambda \in \C^*$ is a root of the polynomial $\calp_{\tilde{\Delta}, \call}$ if and only if $[\lambda : 1]$ is a $\tilde{\varphi}_{a,b}$-special point of $\C\Pp^1$. Consider now a good resolution $\pi : (S_{\pi}, E_{\pi}) \to (S, 0)$ which lifts the indeterminacies of $\Pp \varphi_{a,b}$. It results from Definition \ref{def:dicrit}  that $\pi$ lifts also the indeterminacies of $\Pp \tilde{\varphi}_{a,b}$. As $\alpha, \beta \in \mathfrak{m}_{S,s}$, their pull-backs $\pi^* \alpha$ and $\pi^* \beta$ vanish on $E_{\pi}$. 
    This implies that {\em the restrictions to $E_{\pi}$ of the morphisms $\Pp \varphi_{a,b}$ and $\Pp \tilde{\varphi}_{a,b}$ coincide}. As they are determined by these restrictions, the sets of $\varphi_{a,b}$-special points and of $\tilde{\varphi}_{a,b}$-special points also coincide. The  proposition follows.  
\end{proof}

We want to prove now that $\calp_{\Delta, \call}$ and $\calp_{\tilde{\Delta}, \call}$ not only have the same sets of roots, but that the corresponding multiplicities also coincide. In order to reach this conclusion, we look now at $f$ and $\tilde{f}$ as members of a {\em holomorphic family} of elements of $\mathfrak{m}_{S,s}$, and similarly for $g$ and $\tilde{g}$. Namely, consider:
   \[
       \boxed{f_t} := (1 +  t \alpha) f, \ \ 
       \boxed{g_t} := (1 +  t\beta) g
    \]
for every $t \in \C$. Therefore $f_0 = f, f_1 = \tilde{f}$ and $g_0 = g, g_1 = \tilde{g}$. 

Define $\boxed{\varphi_t} := (f_t, g_t) : (S, s) \to (\C^2, 0)$. Denote by $\boxed{\Delta_t} \hookrightarrow (\C^2, 0)$ the discriminant divisor of $\varphi_t$.

\begin{proposition}   \label{prop:holfamdiv}
   The family $(\Delta_t)_{t \in \C}$ of discriminant divisors is holomorphic.  That is, there 
   exists an effective divisor on $(\C^2, 0)  \times \C$ whose restriction to each germ of 
   surface $(\C^2, 0)  \times \{t\}$ is the divisor $\Delta_t$.
\end{proposition}

\begin{proof}
    Consider the map:
    \[  \Phi := (f_t, g_t, t) :  (S, s) \times \C_t \to (\C^2, 0)  \times \C \]
 It is a germ of finite morphism defined along the curve $\{s\} \times \C$. 
 Consider a finite representative of it and denoted in the same way:  
      \[  \Phi :  S \times \C \to \B  \times \C \]
 where $\B$ is a suitable bidisk neighborhood of the origin of $\C^2$. 
 Restricting $ \Phi $  to the complements of the curves $\{s\} \times \C$ and $\{0\} \times \C$ 
 in the source and respectively the target, we get a finite morphism of $3$-dimensional 
 {\em smooth} complex varieties:
       \[ \Phi^{\circ} : S^{\circ}  \times \C \to \B^{\circ}   \times \C, \]
 where $S^{\circ}  := S \setminus \{s\}$ and  $\B^{\circ}  := \B \setminus \{s\}$. 
 Consider similarly the morphisms 
       \[ \varphi_t^{\circ} : S^{\circ}   \to \B^{\circ}, \  \forall \   t \in \C.  \]
       
    Look now at the Fitting discriminant subspace $\Delta^F_{\Phi}$ of $\Phi$, in the sense of 
    Definition \ref{def:twodiscrloci}. As Fitting discriminant subspaces 
    are compatible with base changes (see Proposition \ref{prop:compatdiscrbc}), 
    $\Delta^F_{\Phi^{\circ}}$ is the restriction of $\Delta^F_{\Phi}$ to $\B^{\circ}   \times \C$. 
    By Proposition \ref{prop:equalspaces} applied to $\Phi^{\circ}$ and to $\varphi_t^{\circ}$, we have:
        \begin{equation} \label{eq:twoeqdiscrphi}
              \Delta^F_{\Phi^{\circ}} = \Delta^D_{\Phi^{\circ}}, \  \  
              \Delta^F_{\varphi_t^{\circ}} = \Delta^D_{\varphi_t^{\circ}}. 
         \end{equation}
     Consider now a fixed $t \in \C$. 
     Applying Proposition \ref{prop:compatdiscrbc} to the base change diagram  
        \[  \begin{tikzcd}    
                      S \times \{t\}   \arrow[r, "h_{S \times \C}"] \arrow[d,  "\varphi_t" ']
                          \arrow[dr, phantom, "\lrcorner", very near end]
                          &  S \times \C \arrow[d, "\Phi"] \\
                      \B   \times \{t\} \arrow[r, "h_{\B \times \C}" ']    &   \B   \times \C  
             \end{tikzcd} ,  \]  
     we get $\Delta^F_{\varphi_t} = h_{\B \times \C}^* \Delta^F_{\Phi}$. As a consequence, 
      $\Delta^F_{\varphi_t^{\circ}} = h_{\B^{\circ} \times \C}^* \Delta^F_{\Phi^{\circ}}$. 
      Combining this with the equalities (\ref{eq:twoeqdiscrphi}), we get 
           $\Delta^D_{\varphi_t^{\circ}} = h_{\B^{\circ} \times \C}^* \Delta^D_{\Phi^{\circ}}$.  
       Taking closures, we get:
            \[ \Delta^D_{\varphi_t} = h_{\B \times \C}^* \Delta^D_{\Phi}. \]
        Indeed, seen as a complex subspace of $\B$, the discriminant curve 
        $\Delta^D_{\varphi_t}$ is the 
       closure of $\Delta^D_{\varphi_t^{\circ}}$ and similarly, $\Delta^D_{\Phi}$ is the 
       closure inside $\B \times \C$ of $\Delta^D_{\Phi^{\circ}}$. 
       But the germ $\Delta_t$ is equal by definition to the germ $\Delta^D_{\varphi_t}$.  
       Therefore, $\Delta_t$ is the restriction of the divisor $\Delta^D_{\Phi}$ 
       to the germ $(\C^2, 0)$. 
       This proves the holomorphic dependence on the parameter $t$ of the family 
       $(\Delta_t)_{t \in \C}$.     
\end{proof}

\begin{proposition}
    \label{prop:samepol}
      The polynomials $\calp_{\Delta, \call}$ and $\calp_{\tilde{\Delta}, \call}$ coincide, 
      up to multiplication by a non-zero constant.
\end{proposition}

\begin{proof}
    By Theorem \ref{thm:Cthm}, the family $(\caln(\Delta_t))_{t \in \C}$ of Newton polygons is 
    constant.  Let $\call$ be a compact edge of this common Newton polygon. 
    By Proposition \ref{prop:invroots}, the polynomials $(\calp_{\Delta_t, \call})_{t \in \C}$ 
    have the same sets of roots in $\C^*$. Moreover, they have the same degree, equal to the 
    integral length of $\call$. As a consequence of Proposition \ref{prop:holfamdiv}, 
    this family of polynomials in one variable is {\em holomorphic}, meaning 
    that they have coefficients which vary holomorphically in $t$. Indeed, 
    for a suitable neighborhood $U_{t_0}$ of every 
    point $t_0 \in \C$, the discriminant divisor $\Delta^D_{\Phi}$ appearing in the proof 
    of Proposition \ref{prop:holfamdiv} is definable in $\B \times U_{t_0}$ by a holomorphic function 
    $D(u,v,t)$, therefore $\Delta_{\tau}$ is definable by $D(u,v,\tau)$, for each $\tau \in U_{t_0}$.
    
    Therefore, the multisets of roots of the polynomials $(\calp_{\Delta_t, \call})_{t \in \C}$ 
    vary also holomorphically. As the corresponding sets are constant, this means that the 
    multiplicities of the roots are also constant. Therefore, the family is constant up to 
    multiplication by elements of $\C^*$. In particular, the polynomials $\calp_{\Delta, \call}$ 
    and $\calp_{\tilde{\Delta}, \call}$ coincide up to multiplication by elements of $\C^*$.
\end{proof}

Recall that the notion of {\em initial Newton curve} $Z(I_D)$ of an effective divisor $D$ on $(\C^2, 0)$ was introduced in Definition \ref{def:invobj}. As a consequence of Proposition \ref{prop:samepol}, we get:

\begin{proposition} \label{prop:mainprep}
    The initial Newton curves of the discriminant divisors $\Delta$ and $\tilde{\Delta}$ coincide.
\end{proposition}

\begin{proof}
   Let $\eta, \tilde{\eta} \in \C\{u,v\}$ be defining series of $\Delta$ and $\tilde{\Delta}$ respectively.
     Combining Propositions \ref{prop:samepol} and \ref{prop:linkrestrasspol}, we deduce that for each compact edge $\call$ of the common Newton polygon of $\Delta$ and $\tilde{\Delta}$, the restrictions $\eta_{\call}, \tilde{\eta}_{\call}$ (see Definition \ref{def:invobj})  coincide up to multiplication by a non-zero constant. Therefore, the same is true of the initial Newton polynomials $\eta_{\caln}, \tilde{\eta}_{\caln}$, which define the curves  $Z(I_{\Delta})$ and $ Z(I_{\tilde{\Delta}})$ respectively. The conclusion follows. 
\end{proof}

\begin{remark} 
 One could think that the discriminant germs $\Delta$ and $\tilde{\Delta}$ of Proposition \ref{prop:mainprep} have moreover the same topological type. This is not always true, as the following example shows. Take $(S,s) := (\C^2, 0)$, with coordinates $(x,y)$. 
 Consider the pair of series $(f, g) :=(x-y^2,x^3+y^3)\in \mathbb C\{x,y\}^2$. The critical divisor $\Xi$  of the finite morphism $\varphi := (f, g) : (\C^2, 0) \to (\C^2, 0)$ is defined by the Jacobian determinant of $\varphi$, which shows that $\Xi = Z(y(y+2x^2))$. It is therefore a union of two tangent smooth branches. Parametrizing them by $t_1 \mapsto (t_1, 0)$ and $t_2 \mapsto (t_2, - 2 t_2^2)$ and composing those parametrizations with $\varphi$, we deduce that $\Delta$ is also a union of two tangent smooth branches.  Consider then the morphism $\tilde{\varphi} := ((1+y)f, g) : (\C^2, 0) \to (\C^2, 0)$, which means that we take $(\alpha, \beta) := (y, 0)$ in equations \eqref{eq:tildefunct}.  
 Its critical divisor $\tilde{\Xi}$, computed as the vanishing divisor of the Jacobian determinant of $\tilde{\varphi}$, is $Z((1+y)y^2-x^2(x- 2y - 3y^2))$. The only compact edge of its Newton polygon joins the points $(0,2)$ and $(0,3)$, which shows that $\tilde{\Xi}$ is a singular branch.  The support of the discriminant divisor $\tilde{\Delta}$ of $\tilde{\varphi}$ is therefore also a singular branch. This shows that the pairs $(f, g)$ and $((1+y)f, g)$ of germs of holomorphic functions determine finite morphisms whose discriminants have different topological types.
\end{remark}

\begin{Atheorem} \label{thm:Athm}
    Let $(S,s)$ be a normal surface singularity and let $f, g \in \mathfrak{m}_{S,s}$ be such that $\varphi = (f, g) : (S,s) \to (\C^2, 0)$ is a germ of finite morphism. Then the initial Newton curve of the discriminant divisor $\Delta_{\varphi}$ depends only on the divisors $Z(f), Z(g)$, up to toric automorphisms of $\C^2$.
\end{Atheorem}

\begin{proof}
    The following proof (which is simple once Proposition 10.4 is known) is analogous to the proofs of \cite[Lemma 6.1]{GGL 15} and \cite[Corollary 5.3]{GGP 22}. 
    Any pair of elements of $\mathfrak{m}_{S,s}$ defining $(Z(f), Z(g))$ is of the form 
    $(A \tilde{f}, B \tilde{g})$, where $(A, B) \in (\C^*)^2$ and $\tilde{f}, \tilde{g}$ verify the relations \eqref{eq:tildefunct}. Consider the map:
      \[\begin{array}{cccc}
           \boxed{\tau_{A,B}} : & \C^2 & \to & \C^2 \\
           & (u,v) & \mapsto & (Au, Bv).
        \end{array}\]
    It is a toric automorphism of $\C^2$. The finite map $\boxed{\tilde{\varphi}_{A, B}} := (A \tilde{f}, B \tilde{g}) : (S, s) \to (\C^2, 0)$ satisfies the relation
     $ \tilde{\varphi}_{A, B} = \tau_{A, B} \circ \tilde{\varphi}$. 
    Therefore $\Delta_{\tilde{\varphi}_{A, B}} = \tau_{A, B} (\Delta_{\tilde{\varphi}})$, which implies the analogous relation $Z(I_{\Delta_{\tilde{\varphi}_{A, B}}}) = \tau_{A,B}(Z(I_{\Delta_{\tilde{\varphi}}}))$ between the initial Newton curves of the discriminant divisors of $\tilde{\varphi}_{A, B}$ and $\tilde{\varphi}$. But $Z(I_{\Delta_{\tilde{\varphi}}}) = Z(I_{\Delta_{\varphi}})$ by Proposition \ref{prop:mainprep}, which implies that 
    $Z(I_{\Delta_{\tilde{\varphi}_{A, B}}}) = \tau_{A,B}(Z(I_{\Delta_{\varphi}}))$.
    The theorem is proved. 
\end{proof}

An immediate consequence of Theorem \ref{thm:Athm} is the fact that the discriminant $\Delta_{\varphi}$ is Newton non-degenerate in the sense of Definition \ref{def:Nnondeg} depends only on the effective divisors $Z(f), Z(g)$. By Proposition \ref{prop:Nndsimpleroots} and Proposition \ref{prop:equivdist} combined with Delgado and Maugendre's Theorem \ref{thm:DelMau}, this property holds if and only if, for each compact edge $\call(a,b)$ of the Jacobian Newton polygon $\caln(\Delta_{\varphi})$, there are as many $\varphi_{a,b}$-special points as the degree of the univariate polynomial $\calp_{\Delta_{\varphi}, \call}$, that is, as the integral length of $\call(a,b)$.

 \medskip
As explained in the introduction, Theorem \ref{thm:Athm} allows to define an invariant of pairs of principal effective divisors without common branches on normal surface singularities:

\begin{definition}
    \label{def:initNcurvepair}
    Let $(S,s)$ be a normal surface singularity and $(D_1, D_2)$ be a pair of principal effective divisors on $(S,s)$ without common branches. The {\bf initial Newton curve} of the triple $(S, D_1, D_2)$ is the initial Newton curve in the sense of Definition \ref{def:invobj} of the discriminant divisor of the finite morphism $(f,g) : (S,s) \to (\C^2, 0)$, where $f, g \in \mathfrak{m}_{S,s}$ are any two defining functions of $D_1, D_2$.  
\end{definition}

For instance, consider the normal surface singularity $(S,s)$ in $\C^3$ defined by an equation of the form $z^n - h(x,y) =0$, with $n \in \Z_{>0}$. Then the initial Newton curve associated to the pair of principal divisors $(Z(x), Z(y))$ on $(S,s)$ is the affine curve in $\C^2$ defined by the equation $h_{\caln}(x,y) = 0$. Indeed, the discriminant divisor of the finite morphism $(x,y) : (S,s) \to (\C^2, 0)$ is $Z(h) \hookrightarrow (\C^2,0)$.


\begin{thebibliography}{00} 

\bibitem{A 75}
     N. A'Campo, \href{https://eudml.org/doc/139622}{\em La fonction z\^eta d'une monodromie}. 
     Comment. Math. Helv. {\bf 50} (1975), 233--248. 

\bibitem{AHV 18} 
  J. M. Aroca, H. Hironaka, J. L.  Vicente, {\em Complex analytic desingularization}. Springer, Tokyo, 2018.    

\bibitem{B 77} 
    R. Bassein, \href{https://eudml.org/doc/163045}{\em On smoothable curve singularities: local methods.} 
    Math. Ann. {\bf 230} (1977), no. 3, 273--277.

\bibitem{BG 80}
    R.-O. Buchweitz, G.-M. Greuel, \href{https://eudml.org/doc/142723}{\em The Milnor number and deformations of complex curve singularities.} Invent. Math. {\bf 58} (1980), no. 3, 241--281.

\bibitem{C 03}  
    E. Casas-Alvero, {\em Discriminant of a morphism and inverse images of 
    plane curve singularities.} Math. Proc. Camb. Phil. Soc. (2003), {\bf 135}, 385--394. 

\bibitem{C 10}
   E. Casas-Alvero, {\em Shadows, discriminant and direct images of plane curve singularities}. Internat. J. Math. {\bf 21} (2010), no. 4, 453--474.
    
\bibitem{CS 21} 
    J. L. Cisneros-Molina, J. Seade,  {\em Milnor's fibration theorem for real and complex singularities.} 
      In  {\em Handbook of geometry and topology of singularities II}, 309--359, Springer, 2021.

 \bibitem{C 94} 
     L. Corry, \href{https://fr.scribd.com/doc/268886747/Dedekind-Eudoxus}{\em La teor\'{\i}a de las proporciones de Eudoxio interpretada por Dedekind.} Mathesis {\bf 10} (1994), no. 1, 1--24. 

 \bibitem{D 94} 
    F. Delgado {\em An arithmetical factorization for the critical point set of some map germs from $\C^2$ to $\C^2$}. In {\em Singularities (Lille, 1991)}, 61--100, London Math. Soc. Lecture Note Ser. {\bf 201}, Cambridge Univ. Press, Cambridge, 1994.
     

\bibitem{DM 03}
    F. Delgado, H. Maugendre, \href{https://www.cambridge.org/core/journals/compositio-mathematica/article/special-fibres-and-critical-locus-for-a-pencil-of-plane-curve-singularities/6E14CC6CE25D02323C064E58F9ED4871}{\em Special fibres and critical locus for a pencil of 
    plane curve singularities.} Compositio Math. (2003) {\bf 136}, 69--87. 

\bibitem{DM 14}
    F. Delgado, H. Maugendre, {\em  On the topology of the image by a morphism of plane curve singularities.} Revista Mat. Complutense {\bf 27} (2014), no. 2, 369--384. 

\bibitem{DM 21}
    F. Delgado, H. Maugendre, \href{https://arxiv.org/abs/1601.01647}{\em Pencils and critical loci on normal surfaces.} 
    Revista Mat. Complutense {\bf 34}  (2021), 691--714.  

\bibitem{DM 24}
    F. Delgado, H. Maugendre, \href{https://arxiv.org/abs/2412.13970}{\em On the image of a curve in a normal surface by a plane projection}.  Rev. Real Acad. Cienc. Exactas Fis. Nat. Ser. A-Mat. {\bf 120}, 38 (2026)

       
\bibitem{Euc} Euclid, \href{https://farside.ph.utexas.edu/books/Euclid/Elements.pdf}{\em Elements of geometry.} Translated from Heiberg's 1883-85 edition by R. Fitzpatrick, 2008. 
    

 \bibitem{F 76} G. Fischer, {\em Complex analytic geometry}. Springer, 1976.     
 
\bibitem{GBGPPP 20}
    E. R. Garc\'{\i}a Barroso,  P. D. Gonz{\'a}lez~P{\'e}rez, P. Popescu-Pampu,  \href{http://de.arxiv.org/pdf/1909.06974.pdf}{\em The combinatorics of plane curve singularities. 
    How Newton polygons blossom into lotuses}.  \emph{Handbook of Geometry and Topology of 
    Singularities I}, 1--150, Springer, 2020.

\bibitem{GGL 15}
    E. Garc\'{\i}a Barroso, J. Gwo\'zdziewicz, A. Lenarcik, \href{https://ergarcia.webs.ull.es/2015-Acta-Math-Hungarica.pdf}{\em Non-degeneracy of the discriminant}.  Acta Math. Hungar. 147 (2015), no. 1, 220–246.


\bibitem{G 84}   
    J. L. Gardies, \href{https://www.persee.fr/doc/rhs_0151-4105_1984_num_37_2_1996}{\em Eudoxe et Dedekind.} Revue d'histoire des sciences {\bf 37}, no. 2 (1984),  111--125. 
    
 \bibitem{GR 79} H. Grauert, R. Remmert, {\em Theory of Stein spaces}. Springer, 1979. 
    

\bibitem{GGP 22} 
    B. Gryszka, J. Gwo\'zdziewicz, A. Parusi\'nski, {\em Initial Newton polynomial of the discriminant.}  
    Bull. London Math. Soc. (2022), {\bf 54}, 1584--1594. 

\bibitem{G 12}  
    J. Gwo\'zdziewicz, \href{https://www.intlpress.com/site/pub/files/_fulltext/journals/mrl/2012/0019/0002/MRL-2012-0019-0002-a009.pdf}{\em Invariance of the Jacobian Newton diagram.} 
    Math. Res. Lett. {\bf 19} no. 2 (2012), 377--382. 

  \bibitem{H 74}
    H. Hironaka, {\em Introduction to the theory of infinitely near singular points.} Memorias de Matem\'atica del Instituto ``Jorge Juan'' {\bf 28}. Consejo Superior de Investigaciones Científicas, Madrid, 1974. 
    
  \bibitem{JP 00}
        T. de Jong, G. Pfister, {\em Local analytic geometry}. Advanced Lectures in Maths. 
        Friedr. Vieweg \& Sohn, 2000.

\bibitem{K 76}
    A. G. Kouchnirenko, \href{https://eudml.org/doc/142365}{\em Poly\`edres de Newton et nombres de Milnor.} Invent. Math. {\bf 32} (1976), no. 1, 1--31. 
    
 \bibitem{K 86} E. Kunz, {\em K\"ahler differentials}. Vieweg, 1986. 

\bibitem{KP 04} 
   T.-C. Kuo, A. Parusi\'nski, \href{https://arxiv.org/abs/math/0211408}{\em Newton-Puiseux roots of Jacobian determinants.} J. Alg. Geom. {\bf 13} (2004), no. 3, 579--601.
    
  \bibitem{L 76}  
     D. T. L\^e, \href{https://www.researchgate.net/profile/Dung-Trang-Le/publication/265678942_Some_Remarks_on_Relative_Monodromy/links/5ef9f10645851550507b27de/Some-Remarks-on-Relative-Monodromy.pdf}{\em Some remarks on relative monodromy}. In {\em Real and complex singularities} (Proc. Ninth Nordic Summer School/NAVF Sympos. Math., Oslo, 1976), 397--403, Sijthoff \& Noordhoff, 1977.

 \bibitem{LMW 01}
    D. T. L\^e, H. Maugendre, C. Weber, \href{https://doc.rero.ch/record/296053/files/63-3-533.pdf}{\em Geometry of critical loci.} J. London Math. Soc. (2) {\bf 63} (2001), no. 3, 533--552.

\bibitem{LMW 89}
     D. T. L\^e,  F. Michel, C. Weber, \href{https://eudml.org/doc/89984}{\em Sur le comportement des polaires associ\'ees aux germes de courbes planes}.  Compositio Math. {\bf 72} (1989), no. 1, 87--113.

  \bibitem{LW 97}
     D.T. L\^e, C. Weber, \href{https://archive-ouverte.unige.ch/unige:12784}{\em \'Equisingularit\'e dans les pinceaux de germes de courbes planes et $C^0$-suffisance.}  Ens. Math. {\bf 43} (1997), 355--380. 
     
     \bibitem{L 02} Q. Liu, {\em Algebraic geometry and arithmetic curves}. Oxford Univ. Press, 2002. 

   \bibitem{M 82}
     P. Maisonobe, \href{http://www.numdam.org/item/10.5802/aif.895.pdf}{\em Lieu discriminant d'un germe analytique de corang 1 de $\C^2,0$ vers $\C^2,0$}.  Ann. Inst. Fourier (Grenoble) {\bf 32} (1982), no. 4, 91--118. 

     

\bibitem{M 99} 
    H. Maugendre, {\em Discriminant of a germ $\Phi:(\C^2,0) \to (\C^2,0)$ and Seifert fibred manifolds.} J. London Math. Soc. (2) {\bf 59} (1999), no. 1, 207--226.

\bibitem{MM 20} 
    H. Maugendre, F. Michel, \href{https://www.journalofsing.org/volume20/maugendre-michel.pdf}{\em On the growth behaviour of Hironaka quotients.}  
    J. Singul. {\bf 20} (2020), 31--53. 

\bibitem{M 18} 
   S. Menn, {\em Eudoxus' theory of proportion and his method of exhaustion.} 
   In  {\em Logic, philosophy of mathematics and their history}, 185--230, Tributes {\bf 36}, 
   Coll. Publ., 2018.

\bibitem{M 08}  
    F. Michel, {\em Jacobian curves for normal complex surfaces.} 
    Contemporary Math. {\bf 475} (2008), 135--150. 

\bibitem{M 68}
    J. Milnor, {\em Singular points of complex hypersurfaces.} Annals of Maths. Studies {\bf 61}. 
    Princeton Univ. Press, 1968.

 \bibitem{MN 20} D. Mond, J. J. Nu\~{n}o-Ballesteros, {\em Singularities of mappings. The local behaviour of smooth and complex analytic mappings.} Springer, 2020.
    
 \bibitem{M 61} D. Mumford, \href{http://www.numdam.org/item/PMIHES_1961__9__5_0.pdf}{\em The topology of normal singularities of an algebraic surface and 
     a criterion for simplicity.} Inst. Hautes \'Etudes Sci. Publ. Math. No. {\bf 9} (1961), 5--22.

 \bibitem{N 91} A. N\'emethi, \href{http://www.numdam.org/item/CM_1991__79_1_63_0.pdf}{\em The Milnor fiber and the zeta function of the singularities of type $f = P(h,g)$.} Compos. Math. {\bf 79} no. 1 (1991), 63--97.     

\bibitem{O 88} T. Oda, {\em Convex Bodies and Algebraic Geometry}. Springer-Verlag, 1988.


\bibitem{PS 13} 
   P. Popescu-Pampu, D. Stepanov,  \href{https://arxiv.org/abs/1204.6154}{\em  Local tropicalization}.   In {\em Algebraic and Combinatorial aspects of Tropical Geometry.}  Proc. Castro Urdiales 2011,  E. Brugall\'e et al. eds., Contemp. Maths. {\bf 589}, AMS, 2013, 253--316.

\bibitem{PS 25} 
   P. Popescu-Pampu, D. Stepanov, \href{https://arxiv.org/abs/2502.11224}{\em An introduction to local tropicalization.}   To appear in Volume {\bf 8} of the {\em Handbook of Geometry and Topology of Singularities}, Springer, 2025.

\bibitem{S 01}
   J. Snoussi, \href{https://ems.press/journals/cmh/articles/316}{\em Limites d'espaces tangents \`a une surface normale}.  Comm. Math. Helv. {\bf 76} (2001), no. 1, 61--88. 

\bibitem{SZ 87} 
   J. Steenbrink, S. Zucker, {\em Polar curves, resolution of singularities and the filtered mixed Hodge structure on the vanishing cohomology.} Lect. Notes in Maths. {\bf 1273}, Springer-Verlag, 1987, 178--202. 

\bibitem{T 97} 
   S. Taher,  {\em Sur le lieu polaire de morphismes analytiques complexes}. C. R. Acad. Sci. Paris S\'er. I Math. {\bf 324} (1997), no. 4, 439--442.

\bibitem{T 73} 
    B. Teissier, \href{http://www.numdam.org/item/AST_1973__7-8__285_0/}{\em Cycles \'evanescents, sections planes et conditions de Whitney.} In {\em Singularit\'es \`a Carg\`ese},   285--362, Ast\'erisque {\bf 7} et {\bf 8}, Soc. Math. France, 1973.

\bibitem{T 77} 
   B. Teissier, \href{https://webusers.imj-prg.fr/~bernard.teissier/documents/Thehunting.pdf}{\em The hunting of invariants in the geometry of discriminants.} In {\em Real and complex singularities} (Proc. Ninth Nordic Summer School/NAVF Sympos. Math., Oslo, 1976), 565--678, Sijthoff \& Noordhoff, 1977.

\bibitem{T 80} 
  B. Teissier, 
{\em Poly\`edre de Newton jacobien et \'equisingularit\'e.} In {\em Seminar on Singularities (Paris, 1976/1977)}, 193--221, Publ. Math. Univ. Paris VII, 7, 1980.

\end{thebibliography}
\end{document}